\documentclass[11pt]{amsart}
\usepackage{amsmath,amssymb}
\usepackage{graphicx}
\usepackage{amsfonts,amscd,latexsym,bbm,epsfig,epic,eepic,oldgerm,psfrag}
\usepackage{a4wide}   

\usepackage{color}
\usepackage{soul}

\definecolor{blue}{rgb}{0,0,1}
\definecolor{red}{rgb}{1,0,0}
\definecolor{green}{rgb}{0,.7,0}

\newcommand{\proofend}{\hspace*{\fill} $\Box$\\}
\newcommand{\diam}{\hspace*{\fill} $\Diamond$}

\def\cc{{\mathcal C}}

\def\cj{{\mathcal J}}

\def\CC{\mathbb{C}}

\def\NN{\mathbb{N}}

\def\QQ{\mathbb{Q}}
\def\RR{\mathbb{R}}

\def\ZZ{\mathbb{Z}}

\def\s{\smallskip}
\def\ni{\noindent}

\def\m{\medskip}

\def\eps{\epsilon}
\def\gve{\varepsilon}

\def\go{\omega}

\def\Vol{\operatorname {Vol\,}}
\def\cHZ{c_{\mbox{\tiny HZ}}}
\def\cG{c_{\mbox{\tiny G}}}

\def\span{\operatorname {span}}
\def\PD{\operatorname {PD}}

\def\pt{\operatorname {pt}}
\def\T{\operatorname {T}}
\def\can{{\operatorname {can}}}
\def\Hom{\operatorname {Hom}}

\def\s{\smallskip}
\def\m{\medskip}

\def\ni{\noindent
}

\newcommand{\Ima}{{\rm Im\,}}
\newcommand{\less}{{\smallsetminus}}

\newcommand{\al}{{\alpha}}

\newcommand{\om}{{\omega}}
\renewcommand{\eps}{{\varepsilon}}

\newcommand{\la}{{\lambda}}
\newcommand{\La}{{\Lambda}}
\newcommand{\si}{{\sigma}}

\newcommand{\ov}{\overline}

\newcommand{\id}{{\rm id}}

\renewcommand{\Tilde}{\widetilde}

\newcommand{\N}{{\mathbb N}}
\newcommand{\Q}{{\mathbb Q}}
\newcommand{\R}{{\mathbb R}}
\newcommand{\C}{{\mathbb C}}

\newcommand{\SL}{{\rm SL}}

\newcommand{\Z}{{\mathbb Z}}

\newcommand{\bx}{{\mathbf x}}
\newcommand{\by}{{\mathbf y}}

\newtheorem{theorem}{Theorem}[section]
\newtheorem{thm}[theorem]{Theorem}
\newtheorem{corollary}[theorem]{Corollary}
\newtheorem{cor}[theorem]{Corollary}
\newtheorem{lemma}[theorem]{Lemma}

\newtheorem{proposition}[theorem]{Proposition}
\newtheorem{prop}[theorem]{Proposition}

\newtheorem{definition}[theorem]{Definition}

\newtheorem{example}[theorem]{Example}

\newtheorem{remark}[theorem]{Remark}
\newtheorem{remarks}[theorem]{Remarks}

\numberwithin{equation}{section}

\newcommand{\MS}{{\medskip}}

\begin{document}

\title{The Gromov width of $4$-dimensional tori}
\author{Janko Latschev} 
\address{(J.~Latschev) 
Fachbereich Mathematik, Universit\"at Hamburg, Bundesstrasse~55, 20146 Hamburg, Germany} 
\email{janko.latschev@math.uni-hamburg.de}
\author{Dusa McDuff} \thanks{partially supported by NSF grant DMS 0905191.}
\address{(D.~McDuff)
Department of Mathematics, 
Barnard College, Columbia University, New York, 
NY 10027-6598, USA.}
\email{dmcduff@barnard.edu}
\author{Felix Schlenk} \thanks{partially supported by SNF grant 200021-125352/1.}
\address{(F.~Schlenk) 
Institut de Math\'ematiques,
Universit\'e de Neuch\^atel, 
Rue \'Emile Argand~11, 
CP~158,
2000 Neuch\^atel,
Switzerland} 
\email{schlenk@unine.ch}
\date{\today}

\begin{abstract} 
Let $\om$ be any linear symplectic form on the $4$-torus $T^4$.  
We show that in all cases $(T^4,\om)$ 
can be fully filled by one symplectic ball.
If $(T^4,\om)$ is not symplectomorphic to a product
$T^2(\mu) \times T^2(\mu)$ of equal sized factors, then it
can also be fully filled by any finite collection of balls
provided only that their total volume is less than that of~$(T^4,\om)$.
\end{abstract}

\maketitle

\tableofcontents
\section{Introduction}

\ni
It has been known since Gromov's paper~\cite{Gr}
that symplectic embedding questions lie at the heart of
symplectic geometry.
For instance, Gromov's Nonsqueezing theorem implies that
for every natural number $k$,
there is no symplectically embedded ball in the product $S^2(k) \times S^2(1)$
of $2$-spheres of areas~$k$ and~$1$ that fills more than $\frac{1}{2k}$ of the volume.
In this paper we study symplectic embeddings of balls 
into $4$-dimensional tori with linear symplectic forms.
Our main result is that the only obstruction to symplectically embedding a 4-ball
into such a manifold is the total volume.

Consider the open ball of capacity~$a$,
$$
B^{2n}(a) \,=\, \Bigl\{ z \in \CC^n \;\Big|\; \pi \sum_{j=1}^n |z_j|^2 < a \Bigr\},
$$
in standard symplectic space $\left( \RR^{2n},\go_0 \right)$,
where $\go_0 = \sum_{j=1}^n dx_j \wedge dy_j$.
The {\it Gromov width} of a $2n$-dimensional symplectic manifold $(M,\go)$, introduced in~\cite{Gr}, 
is defined as 
\begin{equation} \label{def:Gwidth}
\cG( M,\omega) \,=\, 
\sup \left\{ a \mid B^{2n}(a) \mbox{ symplectically embeds into } (M,\omega) \right\}.
\end{equation}
Computations and estimates of the Gromov width 
for various examples can be found 
in~\cite{B1,B2,BC,Gr,Ji,KT,LM,LM1,Lu,M-variants,MP,MSl,Sch-book}.

If the symplectic manifold $(M,\go)$ has finite volume, 
an invariant equivalent to its Gromov width is the {\it ball filling number}
$$
p(M,\go) \,=\, \sup \frac{\Vol \bigl( B^{2n}(a) \bigr)}{\Vol (M,\go)}
$$ 
where the supremum is taken over all balls $B^{2n}(a)$ that symplectically embed into $(M,\go)$,
and where the volume is defined as $\frac{1}{n!} \int_M \go^n$.
Since $\Vol (B^{2n}(a)) = \frac{a^n}{n!}$, 
\begin{equation} \label{e:cGp}
p(M,\omega) \,=\, \frac{(c_G(M,\omega))^n}{n! \Vol (M,\omega)} .
\end{equation}

If $p(M,\go) <1$ one says that {\it there is a filling obstruction},
while if $p(M,\go) =1$ one says that $(M,\go)$ 
{\it admits a full filling by one ball}.\footnote{Our {\it ball filling number}\/ is called 
{\it first packing number}\/ by other authors, and {\it full fillings by one ball}\/ also go 
under the name of {\it full packings by one ball}. 
We refer to~\S\ref{s:quest} for a discussion of {\it full fillings}\/
versus {\it very full fillings}.}
In this paper our main focus is the filling number of $4$-tori with a linear symplectic form~$\om$, 
i.e.\ those which can be identified with the quotient of~$\R^4$, with its standard symplectic structure, 
by a suitable lattice~$\La$. 
We also study other related filling questions in which the ball is replaced by 
a disjoint union of balls. 

Filling obstructions usually come from non-constant holomorphic spheres.
In tori, however, there are no such spheres.
One can thus believe that for tori there should be no filling obstructions.
For the standard torus $T(1,1) := \RR^4/\ZZ^4$, 
there is the obvious lower bound $p(T(1,1)) \ge \frac 12$
coming from the inclusion of the ball $B^4(1)$ into the polydisc $B^2(1) \times B^2(1)$;
see also Figure~\ref{fig.dist} below.  
A better lower bound (namely $p(T(1,1)) \ge \frac 89$) comes from algebraic geometry,
see~\eqref{eq:sesh_bound_standard_4torus} in~\S\ref{ss:Seshadri}.
 We give an explicit realization of this embedding in Example~\ref{ex:2} below.

Our main result is

\begin{theorem} \label{t:1}
Every $4$-dimensional linear symplectic torus admits a full filling by one ball;
in other words, $p(T^4,\om) = 1$ for all linear $\om$.
\end{theorem} 


The symplectic (resp.\ K\"ahler) cone of a smooth 
oriented manifold~$X$ is the 
set of cohomology classes $\alpha \in H^2(X;\R)$ that can be represented 
by a symplectic (resp.\ K\"ahler) form,
where here we consider 
symplectic forms that are 
compatible with the given orientation on~$X$
(resp. K\"ahler forms that are compatible with any complex structure giving this orientation).
The symplectic cone~$\cc(T^4)$ of $T^4$ with a given orientation 
is $\{ \alpha \in H^2(T^4;\RR) \mid \alpha^2 > 0 \}$.
Each such class has a linear representative. 
From Theorem~\ref{t:1}, we get the following characterization of the symplectic cone of the 1-point blow up of a given oriented torus~$T^4$.

\begin{corollary} \label{c:cone}
Denoting by $E \in H_2(\Tilde X;\ZZ)$ the homology class of the exceptional divisor 
(with some orientation) in~$\Tilde X = T^4 \sharp \overline{\CC P^2}$, 
the symplectic cone of~$\Tilde X$ is
$$
\cc(\Tilde X) \,=\, 
\left\{ \alpha \in H^2(\Tilde X;\RR) \mid \alpha^2>0,\, \alpha(E) \ne 0 \right\} .
$$
\end{corollary}

While there are many examples of non-K\"ahler symplectic manifolds, 
it is much harder to find K\"ahler manifolds 
for which the K\"ahler and symplectic cones differ.  
Some examples are given by Dr\u aghici~\cite{Dra} and Li--Usher~\cite{Li-Usher}.
More recently, Cascini--Panov~\cite{CP} showed that the K\"ahler and symplectic cones differ for the one point blow-up of $T^2 \times S^2$. 
With the help of Corollary~\ref{c:cone} we obtain 
another simple example.

\begin{corollary} \label{c:1}
Let $\Tilde X$ be the blow-up $T^4 \sharp \overline{\C P^2}$ of the $4$-torus
in one point. 
Then the symplectic cone of~$\Tilde X$ is strictly bigger than the K\"ahler cone.
\end{corollary}

\m
\bigskip \ni
{\bf Outline of the proof of Theorem~\ref{t:1}}

\m
\ni
As we will see, for our purposes linear $4$-tori divide
into three classes: 
the {\it standard torus} $T(1,1)$ 
(and its rescalings), 
{\it all other rational tori} (in which $[\om]$ is a multiple of a rational class), 
and {\it irrational tori} (in which the image of the homomorphism 
$\int \om \colon H_2(M;\Z) \to \R$ has rank at least~$2$
over $\Q$).

It turns out that every rational torus is 
(up to scaling)
symplectomorphic to a product torus $T^2(d_1) \times T^2(d_2)$,
where $d_1,d_2 \in \NN$ denote the areas of the two factors; see Lemma \ref{lem:type}. 
Thus the family of product tori $T(1,\mu): = T^2(1) \times T^2(\mu)$
with $\mu \ge 1$ 
contains all rational tori,
up to scaling. 
With this in mind, our proof proceeds as follows.

\medskip \ni
{\it 1. Linear algebra.}
By a simple symplectic linear algebra argument, the tori $T(m,n)$ and $T(1,mn)$
are symplectomorphic for relatively prime integers $m,n$ (see Remark~\ref{r:normalform}).
Hence:

\begin{lemma} \label{l:linalg}
$p \bigl( T(\frac mn,1) \bigr) = p\bigl(T(1,m n)\bigr)$ for $m,n \in \NN$ relatively prime.
\end{lemma}

\ni
{\it 2. Algebraic geometry.}
Buchdahl~\cite{Buchdahl} and Lamari~\cite{Lamari}
found a condition on a cohomology class $\alpha \in H^{1,1}(X;\R)$ on some complex surface~$X$ that guarantees the existence of a K\"ahler representative of~$\alpha$.
We shall verify this condition on blow-ups of irrational tori to obtain:

\begin{proposition} \label{p:irrat}
$p(T^4,\om)=1$ for all irrational linear tori $(T^4,\om)$.
\end{proposition}

\ni
{\it 3. Full fillings of $T^2(1) \times S^2(\mu)$.}
Denote by $S^2(\mu)$ the $2$-sphere endowed with an area form of area~$\mu$.
Biran~\cite{B1} proved that $T^2(1) \times S^2(\mu)$ 
can be fully packed by one ball
provided that $\mu \ge 2$.
We shall show that such an almost filling ball can be made to lie in the complement of
a constant section $T^2(1) \times \pt$. 
Since the open disc bundle 
$T^2(1) \times D^2(\mu) = \bigl( T^2(1) \times S^2(\mu) \bigr) \less \bigl( T^2(1) \times \pt \bigr)$ symplectically embeds into~$T^2(1) \times T^2(\mu) = T(1,\mu)$, 
we obtain

\begin{proposition} \label{p:muge2}
$p\bigl(T(\mu,1)\bigr) =1$ for all~$\mu \ge 2$.
\end{proposition}

\begin{corollary} \label{c:mu}
$p\bigl(T(\mu,1)\bigr) =1$ for all~$\mu \ne 1$.
\end{corollary}

\begin{proof}
In view of Propositions~\ref{p:irrat} and \ref{p:muge2} we 
need only consider $\mu \in (1,2) \cap \QQ$.
If we write $\mu = \frac mn$ with $m,n \in \NN$ relatively prime,
then 
$m>n \ge 2$ , giving $mn \ge 6$.
Hence Lemma~\ref{l:linalg} and Proposition~\ref{p:muge2} imply
$p\bigl(T(\mu,1)\bigr) = p\bigl(T(1,mn)\bigr) =1$.
\end{proof}

\medskip \ni
{\it 4. A symplectic embedding construction.}
The only case not covered by the above discussion is the standard product torus $T(1,1)$. 
To prove $p(T(1,1)) = 1$ we shall construct for each ball $B^4(a)$ of volume $\frac{a^2}2 < 1$ 
an explicit symplectic embedding into~$T(1,1)$.
Fix $a<\sqrt 2$.
We start with an almost full embedding $B^4(a) \to \Diamond \times \Box$,
where $\Diamond \subset \RR^2(x_1,x_2)$ is a diamond-shaped domain
(see Figure~\ref{fig.dist}~(I) below), 
and $\Box = (0,1)^2 \subset \RR^2(y_1,y_2)$.
The main step is then to construct a symplectic embedding 
$\Diamond \times \Box \to \RR^4$ with image~$U$
such that the projection $\RR^4 \to T(1,1)=\RR^4/\ZZ^4$ is injective on~$U$.
%

\smallskip
The resulting embedding $B^4(a) \to T(1,1)$ uses all four homological directions of~$T(1,1)$.
This must be so.
Indeed, assume that there exists an embedding $B^4(a) \to T(1,1)$ that factors, 
for instance, as
$$
B^4(a) \,\stackrel{\psi}{\longrightarrow}\, T^3(x_1,y_1,x_2) \times (0,1) \,\to\, T(1,1)
$$
with $(0,1) \subset \RR(y_2)$.
It is easy to see that there exists a symplectic embedding~$\rho$ of the annulus 
$T^1(x_2) \times (0,1)$ into $B^2(1)\subset \RR^2(x_2,y_2)$.
Composing $\psi$ with $\id \times \rho$ we obtain a symplectic embedding of $B^4(a)$ into 
$T^2(x_1,y_1) \times B^2(1)$, which lifts to $\RR^2(x_1,y_1) \times B^2(1)$.
Hence $a \le 1$ by the Nonsqueezing Theorem.
A similar discussion applies to all sufficiently large balls in product tori~$T(\mu,1)$ 
with~$1\le \mu <2$.

\begin{remark} \label{rem:construction}\rm
Parts of the above construction yield an explicit full filling by one ball of 
the $4$-torus~$T(\mu,1)$  
for all $\mu = \frac{2m^2}{n^2}$ with $m,n$ relatively prime.
Since the set of rational numbers~$\mu$ of this form is dense in~$\RR_{>0}$,
one is tempted to derive $p \bigl(T(1,1)\bigr) = 1$ from $p \bigl(T(\mu,1)\bigr) =1$ for~$\mu>1$ 
by a limiting argument, 
or to derive $p \bigl(T(\mu,1)\bigr) = 1$ for all $\mu \ge 1$ from the elementary explicit full fillings 
of~$T(\mu,1)$ for $\mu = \frac{2m^2}{n^2}$.
However, without further knowledge about the underlying embeddings, 
the function $\mu \mapsto p \bigl(T(\mu,1)\bigr)$ has no obvious continuity properties.
\diam
\end{remark}

\ni
{\bf Filling by more than one ball.}
The general ball packing problem for a symplectic $4$-manifold $(M,\omega)$ is:
Given a collection 
$\overline B\,\!^4(a_1), \dots, \overline B\,\!^4(a_k)$ of closed balls,
does there exist a symplectic embedding of
$\coprod_{j=1}^k \overline B\,\!^4(a_j)$ into $(M,\omega)$?
Since symplectic embeddings are volume preserving, a necessary condition is 
$\Vol \bigl( \coprod_{j=1}^k \ov  B\,\!^4(a_j) \bigr) < \Vol (M,\omega)$.
We prove that for all linear tori {\em except possibly~$T(1,1)$} 
this is the only condition.

\begin{theorem} \label{t:2}
Assume that $(T^4,\om)$ is a linear torus.
Let $\ov B\,\!^4(a_1), \dots, \ov B\,\!^4(a_k)$ be a collection of balls such that
$$
\Vol \Bigl( \coprod_{j=1}^k \ov B\,\!^4(a_j) \Bigr) \,<\, \Vol \left( T^4,\om \right).  
$$
\begin{itemize}
\item[(i)]
If $(T^4,\om)$ is not symplectomorphic to $T(\mu,\mu)$ for some~$\mu>0$,
there exists a symplectic embedding of $\coprod_{j=1}^k \ov B\,\!^4(a_j)$ into~$(T^4,\om)$.

\s
\item[(ii)]
If $(T^4,\om)$ is symplectomorphic to $T(\mu,\mu)$ for some~$\mu>0$,
then $\coprod_{j=1}^k\ov B\,\!^4(a_j)$ symplectically embeds into~$T(\mu,\mu)$ 
under the further restriction that  $a_j < \mu$ for all~$j$.
\end{itemize}
\end{theorem} 

Notice that Theorem~\ref{t:2} generalizes Propositions~\ref{p:irrat} and \ref{p:muge2}.
The extra condition in~(ii) is presumably not needed, 
but the only way we can see to prove this would be by 
explicitly constructing suitable embeddings.

Other examples of manifolds for which the volume is the only obstruction 
to a symplectic embedding of a collection of balls were found by Biran in~\cite{B1,B2}. 
Biran also proved in~\cite{B2} that $T(1,1)$ can be fully packed with $k$~equal balls 
for any~$k \ge 2$.\footnote{In fact, his argument also proves the claim in Theorem~\ref{t:2} 
concerning~$T(1,1)$. His proof is much the same as ours in that he reduces the problem to packing some ruled $4$-manifold. However he considers the projectivization of a holomorphic line bundle of Chern class~$2$ over a genus~$2$ surface, while we use a trivial bundle over~$T^2$. In both cases the spherical fibers have area~$1$.}

\begin{remarks}
\rm
(i)
Our results may give the impression that symplectic embeddings of balls  
into $4$-dimensional tori are as flexible as volume preserving embeddings.
This is far from true, 
as the following consideration from~\cite{B2} shows:
By our results above,
the standard product torus~$T(1,1)$ admits symplectic embeddings
of the ball~$B(a)$ and of the 
disjoint union of two equal size balls~$B(b) \sqcup B(b)$
whenever there is no volume obstruction. 
However, as is well known, 
a symplectic embedding of $B(b) \sqcup B(b)$ into~$B(a)$ 
covers at most half of the volume.
Therefore, given symplectic embeddings $\varphi \colon B(a) \to T(1,1)$ and
$\psi \colon B(b) \sqcup B(b) \to T(1,1)$ that cover more than half of the volume,
it cannot be that the image of $\varphi$ contains the image of~$\psi$.
This ``hidden rigidity'' phenomenon for symplectic embeddings of balls into tori 
clearly does not exist for volume preserving embeddings of balls into tori.

\m
(ii)
Another important invariant of a symplectic manifold $(M,\go)$
is its Hofer--Zehnder capacity $\cHZ (M,\go)$, which is of dynamical nature.
We refer to the books~\cite{HZ, MS} for the definition and elementary properties.
The value of this capacity is unknown for product tori; in fact it is an outstanding problem
to decide whether it is finite or infinite for product tori.

Our computations of the Gromov width $\cG$ of tori give lower bounds for~$\cHZ$,
because $\cG (M,\go) \le \cHZ (M,\go)$ for all symplectic manifolds.
These lower bounds are, however, weaker than the known ones.
These come from the elementary inequality
$$
\cHZ (M,\go) \,\ge\, \cHZ (P,\go_P) + \cHZ (Q,\go_Q),
$$
holding for all products $(M,\go) = (P \times Q, \go_P \oplus \go_Q)$ of closed symplectic manifolds,
together with the fact that the Hofer--Zehnder capacity of a 2-dimensional connected symplectic manifold 
equals its area.
To be explicit, our main theorem implies that $\cHZ \bigl(T(1,1)\bigr) \ge \cG \bigl(T(1,1)\bigr)= \sqrt 2$, 
while it is known that $\cHZ \bigl(T(1,1)\bigr) \ge 1+1 =2$.
\diam
\end{remarks}

The paper is organized as follows.
In Section~\ref{s:ag} we review the lower bounds for the ball filling number 
of $4$-dimensional symplectic tori coming from known computations of Seshadri constants.  
Section~\ref{ss:irrat} contains a proof of Theorem~\ref{t:2}    
in the irrational case.
This proof is based on the construction in Section~\ref{ss:no} 
of symplectic tori with no holomorphic curves. 
In Section~\ref{s:muge2} we prove Theorem~\ref{t:2} for product tori  
$T(1,\mu)$, $\mu \ge 1$, under the condition that $\min \{ a_j,b_j \} < \mu$ for all~$j$. 
In Section~\ref{s:basic} we explain the embedding construction
that we use in Section~\ref{s:t1} to prove 
$p\bigl(T(1,1)\bigr) =1$,
completing the proof of Theorem~\ref{t:1}.
In Section~\ref{s:cone} we prove Corollaries~\ref{c:cone} and \ref{c:1}, and
in Section~\ref{s:quest} we state some open problems related to filling tori.

\medskip
\ni
{\bf Acknowledgment.}
This work has its origin in discussions between the authors and Dietmar Salamon at the Edifest~2010, 
and we would like to thank ETH Z\"urich and its FIM for the stimulating atmosphere during the conference. 
We also thank Paul Biran, Dietmar Salamon and Sewa Shevchishin 
for fruitful discussions.
Finally, we thank the referee for a careful reading and helpful comments on the exposition.

\section{Relations to algebraic geometry} \label{s:ag}

\ni
In this section, we review the implications of some results in algebraic geometry 
for the Gromov radius of $4$-dimensional symplectic tori, 
and also of some higher dimensional ones.

\subsection{Basics}
Before discussing the complex geometry of tori, we recall a classical result.
\begin{lemma}\label{lem:type}
Suppose $\omega$ is a linear symplectic form on a torus 
$T= \R^{2n}/\Lambda$ with integral cohomology class. 
Then $(T,\omega)$ is symplectomorphic to a product of $2$-dimensional tori
$$
T^2(d_1) \times \dots \times T^2(d_n)
$$
with symplectic areas $d_j>0$ satisfying $d_j|d_{j+1}$ for all $j=1,\dots,n-1$.
Moreover, the sequence $d_1 | d_2 | \dots | d_n$ is uniquely determined by~$\Lambda$. 
\end{lemma}

\begin{remark} \label{r:normalform}
{\rm
It follows that a 4-dimensional product torus $T(m,n) = T^2(m) \times T^2(n)$ 
with integer areas~$m$ and~$n$ is symplectomorphic to 
$T(g,\ell)$ where $g=\gcd(m,n)$, $\ell = \operatorname{lcm}(m,n)$.
}
\end{remark}

\ni
{\it Proof of Lemma~\ref{lem:type}.}
Since $\omega$ is linear, 
it lifts to a linear symplectic form on~$\R^{2n}$, which we again denote by~$\omega$. 
The fact that it represents an integral cohomology class 
on~$T$ is equivalent to 
the fact that it takes integer values when restricted to $\Lambda \times \Lambda$. 
Denote by $d_1 \in \Z$ the positive generator of this image subgroup, 
and choose $e_1,f_1 \in \Lambda$ with $\omega(e_1,f_1)=d_1$.

Every lattice point $\lambda \in \Lambda$ can be written as
$$
\lambda \,=\, \frac {\omega(\lambda,f_1)}{d_1}\, e_1 + 
              \frac {\omega(e_1,\lambda)}{d_1}\, f_1 + \lambda',
$$
where the coefficients of $e_1$ and $f_1$ are integers by the choice of~$d_1$, 
and where $\lambda' \in \Lambda$ is $\omega$-orthogonal to both~$e_1$ and~$f_1$.
In other words, $\Lambda = \span_\Z (e_1,f_1) \oplus \Lambda'$ for some lower dimensional 
sublattice $\Lambda' \subset \Lambda$. Now repeat the argument with~$\Lambda'$ 
in place of~$\Lambda$, noting that the image of~$\omega$ when restricted to 
$\Lambda' \times \Lambda'$ must be a subgroup of $d_1 \Z \subset \Z$. 
This finishes the proof in $n$~steps. 

To prove the uniqueness of the sequence $d_1|\dots|d_n$ for a given
torus $T=\RR^{2n}/\Lambda$, note that since $\omega$ is non-degenerate and integral, it gives rise to an embedding $\phi: \Lambda \to \Hom(\Lambda,\ZZ)$, namely $\phi(\lambda_1)(\lambda_2)=\omega(\lambda_1,\lambda_2)$. Now the $d_j$ are the torsion coefficients of the finitely generated abelian group 
$\Hom(\Lambda;\ZZ)/\Ima \phi$, which are well-known to be invariants of this group.
\proofend

Complex tori are often defined as the quotients of~$\C^n$ by some cocompact lattice $\Lambda\cong \Z^{2n}$. 
In dimension~$4$, the Enriques--Kodaira classification of compact complex surfaces 
implies that every complex manifold diffeomorphic to~$T^4$ is biholomorphic to such 
a model. In higher dimensions, this is still true if the complex structure 
is compatible with a K\"ahler form, but false in general (for examples, see e.g.\ \cite{COP} and references therein).

Conversely, the standard symplectic form on $\C^n$ descends to a K\"ahler form 
on any quotient $\C^n/\Lambda$, so every complex torus admits 
a compatible K\"ahler structure whose symplectic form is translation invariant.

\subsection{Seshadri constants of tori} \label{ss:Seshadri}

Here we review some results described by Lazarsfeld in~\cite[Chapter~5]{Lazarsfeld:positivityI}, 
which do not seem to be widely known among symplectic geometers. 
For an irreducible projective variety $X$ and a point~$x \in X$ we denote by 
$$
\pi \colon \widetilde X \to X
$$
the blow-up of $X$ at $x$, with exceptional divisor $\Sigma \subset \widetilde X$. 
Recall that a line bundle~$L$ on~$X$ is called {\em nef}\/ 
if for every irreducible curve $C \subset X$ one has $\int_C c_1(L) \geq 0$.
\begin{definition}
{\rm 
(cf.\ \cite[Def.~5.1.1.]{Lazarsfeld:positivityI})
Suppose $L$ is a nef line bundle on~$X$.
The {\em Seshadri constant of $(X,L)$ at $x \in X$} is defined to be the real number 
\begin{equation} \label{def:Seshadri}
\eps(L;x) \,:=\, 
\max \bigl\{ \eps \geq 0 \,\mid\, \int_{\widetilde C}\pi^*(c_1L) - \eps \Sigma \cdot \widetilde C \geq 0 
             \mbox{ \, for all curves } \widetilde C \subset \widetilde X \bigr\}.
\end{equation}
}
\end{definition}
It is clear that $\eps(L,x)$ is always nonnegative, 
and in fact one has the alternative description 
(cf.\ \cite[Prop.~5.1.5.]{Lazarsfeld:positivityI})
\begin{equation}\label{eq:seshadri1}
\eps(L;x) \,=\, \inf_{x \in C \subset X} \frac {\int_C c_1(L)}{\operatorname{mult}_x C},
\end{equation}
where the infimum is taken over all irreducible curves $C \subset X$ passing through~$x$, 
and $\operatorname{mult}_x C \in \N$ denotes the multiplicity of~$C$ at~$x$. 
This shows that one can obtain upper bounds on~$\eps(L;x)$ from specific curves 
passing through $x \in X$. 

From the symplectic point of view, we are particularly interested in the case when~$X$ 
is a smooth projective variety, and $L$~is an ample line bundle.
Then one can choose a K\"ahler form~$\omega_L$ representing~$c_1(L)$. 
Since the space of symplectic forms in a fixed cohomology class which are compatible 
with a fixed (almost) complex structure is contractible, 
any two such forms are symplectically isotopic. 

Now there is a strong relationship between symplectic embeddings of balls and 
symplectic blow-up, which was first described by McDuff~\cite{M-blowup} 
and McDuff--Polterovich~\cite{MP}: 
An embedding of a closed symplectic ball~$B(a)$ of capacity~$a$ 
into a given symplectic manifold~$X$ gives rise to a symplectic form on the 
topological blow-up 
$\pi \colon \Tilde X \to X$ whose cohomology class is given by 
$\pi^*[\omega] - a \PD[\Sigma]$, where $\Sigma \subset \Tilde X$ is the exceptional divisor,
and $\PD [\Sigma]$ denotes the Poincar\'e dual of~$[\Sigma]$.
Conversely, given a tame  symplectic form 
on the complex blow-up $(\Tilde X, \Tilde J)$
in a class $\pi^*\alpha - a \PD [\Sigma]$, 
one can find a symplectically embedded ball~$B(a)$ in $(X,\omega)$ 
with $[\omega] = \alpha \in H^2(X;\R)$,
see~\cite[Cor.~2.1.D]{MP}.

As pointed out in \cite[Thm.~5.1.22.]{Lazarsfeld:positivityI}, this discussion then leads to the following result, 
which is a direct consequence of~\cite[Cor.~2.1.D]{MP}:
\begin{prop} 
\label{prop:seshadribound}
For fixed $X$ and $L$ as above, denote by 
$$
\eps(X,L) := \max_{x \in X} \,\eps(L;x).
$$ 
Then the Gromov width of~$(X,\omega_L)$, defined in~\eqref{def:Gwidth}, satisfies
$$
\cG(X,\omega_L) \,\geq\, \eps(X,L).
$$
\end{prop}

By the relation~\eqref{e:cGp}, this estimate is equivalent to
\begin{equation} \label{est:pe}
p(X,\omega_L) \,\geq\, \frac{(\eps(X,L))^n}{n! \Vol (X,\omega_L)} .
\end{equation}
The proof of Proposition~\ref{prop:seshadribound} is based on the fact that when $\eps(X;L)>0$, 
then the pullback $\widetilde L$ of~$L$ to the blow-up~$\widetilde X$ is ample, 
so that $c_1(\widetilde L)$ has a K\"ahler representative.

\begin{remark} \label{rem:Kahler.eps}
{\rm 
The same blow-up argument also shows that the capacity of the largest \emph{symplectically and holomorphically} embedded ball in the K\"ahler manifold $(X,\omega_L)$ bounds the Seshadri constant~$\eps(X,L)$ from below 
(for details, cf.~\cite[Prop.~5.3.17]{Lazarsfeld:positivityI}).
}
\end{remark}

In what follows, we will study the family of symplectic product tori  
$T^2(1) \times T^2(d)$ where $d \in \Z$. 
By Lemma~\ref{lem:type}, up to rescaling this class contains all symplectic 4-tori 
whose symplectic form is linear and has a rational cohomology class. 
Now suppose that $(T,\omega)$ is such a symplectic torus, 
and choose a translation-invariant compatible complex structure~$J$, 
so that $(T,J,\omega)$ is a K\"ahler manifold. 
If~$L$ is the complex line bundle on~$(T,J)$ with first Chern class~$[\omega]$, 
then $L$ is ample. Complex tori admitting such a line bundle are called abelian varieties, and the line bundle or its first Chern class is called a {\it polarization}. 
Note that, conversely, the first Chern class of any ample line bundle~$L$ on some 
complex torus can be represented by a translation-invariant rational symplectic form~$\omega_L$, 
and so all abelian varieties arise as above. 

The sequence of integers $(d_1,\dots,d_n)$ for $(T,\omega_L)$ appearing in 
Lemma~\ref{lem:type} is called the {\it type}\/ of the polarization, 
and the polarization is called {\it principal}\/ and often denoted by~$\Theta$ 
if it is of type $(1,\dots,1)$, i.e.\ it corresponds to the standard symplectic product torus.

Since translations act transitively on any abelian variety~$A$, 
the Seshadri constants for abelian varieties do not depend on the choice of 
the point~$x \in A$. 
One has the general bounds 
\begin{equation} \label{eq:seshadri2}
d_1 \,\leq\, \eps(A^n,L) \,\leq\, (n!d_1\cdots d_n)^{\frac 1 n}
\end{equation}
for an ample line bundle $L$ of type $(d_1,\dots,d_n)$. 
The upper bound follows from the estimate~\eqref{est:pe},
see also~\cite[Prop.~5.1.9]{Lazarsfeld:positivityI}.
For the lower bound, recall from Lemma~\ref{lem:type} that $d_1 | d_2 | \dots | d_n$,
and consider the ample line bundle~$L'$ of type $(1, d_2/d_1, \dots, d_n/d_1)$.
By \cite[Ex.~5.3.10]{Lazarsfeld:positivityI},
$\eps(A,L') \ge 1$, and hence $\eps(A,L) \ge d_1$.
%
%
%
%
Similarly, the symplectic embedding of the ball of capacity~$d_1$ into the polydisk 
$B^2(d_1) \times \dots \times B^2(d_n) \subset A$ gives the same lower bound~$d_1$ 
for the Gromov width of~$(A,\omega_L)$.

The best lower bounds on Seshadri constants for abelian varieties of a given type 
seem to come from irreducible ones, i.e., those which cannot be written as a product 
of lower-dimensional complex tori. 
Here we list the known results, according to \cite[Rem.\ 5.3.12]{Lazarsfeld:positivityI}.

First, to get a bound on the ball filling number of~$T(1,1)$, 
according to the discussion above we need to consider principally polarized 
abelian surfaces~$(A^2,\Theta)$. 
Steffens~\cite[Prop.~2 and~3]{Steffens} 
has shown that in this case
\begin{equation} \label{eq:seshadri_standard_4torus}
\eps(A^2,\Theta) \,\le\, \tfrac 43,
\end{equation}
with equality if $A$ is irreducible. 
Together with the estimate~\eqref{est:pe} 
we obtain the lower bound 
\begin{equation} \label{eq:sesh_bound_standard_4torus}
p(T(1,1)) \,\geq\, \tfrac 8 9.
\end{equation}

For tori of type $(1,d)$ one can get lower bounds from non-principal polarizations 
of abelian surfaces~$(A^2,L)$. Indeed, it is known from the work of 
Steffens~\cite[Prop.~1]{Steffens} that {\em if $2d$ is a perfect square}, 
then there are abelian surfaces with a polarization~$L$ of type~$(1,d)$ and
\begin{equation} \label{eq:seshadri_2d_square}
\eps(A^2,L) = \sqrt{2d}, 
\end{equation}
which is optimal since it equals the volume bound in~\eqref{eq:seshadri2}.
This immediately implies
\begin{equation} \label{eq:seshbound_2d_square}
p(T(1,d)) = 1 \quad \text{\rm if $2d$ is a perfect square}.
\end{equation}
We will describe explicit examples of such full fillings by one symplectic ball 
in Section~\ref{ss:shears}.
The identities~\eqref{eq:seshbound_2d_square} and Remark~\ref{r:normalform}
imply that $p(T(\mu,1)) = 1$ for all $\mu = \frac{2m^2}{n^2}$ with $m,n$ relatively prime integers.

On the other hand, when $2d$ is not a perfect square, 
then Bauer and Szemberg~\cite{Bauer:abelian} have shown that
\begin{equation} \label{eq:seshadri_2d_not_square}
\eps(A^2,L) \,\leq\, 2d \frac {k_0}{\ell_0} \,=\, \sqrt{2d} \,\cdot \, \sqrt{\frac {2dk_0^2}{2dk_0^2+1}},
\end{equation}
where $(k_0,\ell_0)$ is the smallest solution in positive integers of Pell's equation
$$
\ell^2 -2dk^2 = 1.
$$
(There always exists such a solution, as was first shown by Lagrange, \cite{Lagrange}.)
Moreover, by a result of Bauer~\cite{Bauer:surfaces} equality holds whenever 
positive multiples of~$L$ are the only ample line bundles on~$A$. 
Since complex structures~$J$ with this property exist for all symplectic types~$(1,d)$, 
this gives the best constant for use in Proposition~\ref{prop:seshadribound}. 
For $d \leq 30$, the relevant solutions to Pell's equation have been tabulated 
in~\cite[p.~572]{Bauer:surfaces}, and we give their translation in terms of 
the lower bound on the ball filling numbers 
$p(T(1,d)) \geq \frac {\eps^2}{2d} = \frac {\ell_0^2-1}{\ell_0^2}$ in the following table. 

$$
\begin{array}{|c|ccc|c|ccc|c|ccc|}
\hline
d & k_0 & \ell_0 & \frac {\eps^2}{2d} & d & k_0 & \ell_0 & \frac {\eps^2}{2d}& d & k_0 & \ell_0 & \frac {\eps^2}{2d}\\
\hline
1 & 2 & 3 & \frac 8 9 & 11 & 42 & 197 & \frac {38808}{38809} & 
21 & 2 & 13 & \frac {168}{169}\\
2 & & & 1 & 12 & 1 & 5 & \frac {24}{25} &
22 & 30 & 199 & \frac {39600}{39601}\\
3 & 2 & 5 & \frac {24}{25} & 13 & 10 & 51 & \frac {2600}{2601} &
23 & 3588 & 24335 & \frac {592192224}{592192225}\\
4 & 1 & 3 & \frac 8 9 & 14 & 24 & 127 & \frac {16128}{16129} &
24 & 1 & 7 & \frac {48}{49}\\
5 & 6 & 19 & \frac {360}{361} & 15 & 2 & 11 & \frac {120}{121} &
25 & 14 & 99 & \frac {9800}{9801}\\
6 & 2 & 7 & \frac {48}{49} & 16 & 3 & 17 & \frac {288}{289} &
26 & 90 & 649 & \frac {421200}{421201}\\
7 & 4 & 15 & \frac {224}{225} & 17 & 6 & 35 & \frac {1224}{1225} &
27 & 66 & 485 & \frac {235224}{235225}\\
8 & & & 1 & 18 & & & 1 & 28 & 2 & 15 & \frac {224}{225}\\
9 & 4 & 17 & \frac {288}{289} & 19 & 6 & 37 & \frac {1368}{1369} &
29 & 2574 & 19605 & \frac {384356024}{384356025}\\
10 & 2 & 9 & \frac {80}{81} & 20 & 3 & 19 & \frac {360}{361} &
30 & 4 & 31 & \frac {960}{961}\\
\hline
\end{array}
$$

\medskip
\subsection{Seshadri estimates for higher dimensional tori} \label{ss:higherdim}

One well-studied class of principally polarized abelian varieties of arbitrary dimension 
are the Jacobians of curves (cf.\ e.g.~\cite[Chapter~11]{Birkenhake}). 
Here we just recall that the Jacobian of a complex curve~$C$ 
is the complex torus 
$$
JC \,:=\, \Hom(\Omega^{1,0},\C)/H_1(C;\Z),
$$
where $\Omega^{1,0}$ denotes the complex vector space of holomorphic 1-forms, 
and the embedding $H_1(C;\Z) \subset \Hom(\Omega^{1,0},\C)$ is given by 
integration over cycles. The complex dimension of~$JC$ equals the genus of~$C$, 
and the principal polarization is derived from the natural symplectic structure 
on $H_1(C;\Z) \otimes \R$ which is given by the intersection product.

In complex dimension $n=3$, Bauer and Szemberg \cite{BauerSzemberg:threefolds} 
have shown that a principally polarized abelian variety $(A^3,\Theta)$ 
has $\eps(\Theta) = \frac 3 2$ if $A$ is the Jacobian of a hyperelliptic curve 
of genus~$3$ and
\begin{equation}\label{eq:seshadri_dim3}
\eps(A^3,\Theta) = \tfrac {12} 7
\end{equation}
otherwise. 
(A complex curve is called hyperelliptic if it admits a double branched cover to~$\C P^1$).
Hence $p\bigl(T(1,1,1)\bigr) \ge \frac{288}{343}$.

In complex dimension $n=4$, Debarre~\cite{Debarre:higher} has shown that for 
the Jacobian~$A^4=JC$ of a non-hyperelliptic curve of genus~$4$ one has 
\begin{equation} \label{eq:seshadri_dim4}
\eps(A^4,\Theta) = 2.
\end{equation}
Hence $p\bigl(T(1,1,1,1)\bigr) \ge \frac 23$.

For high dimensions, Jacobians appear to give very poor lower bounds for use in 
Proposition~\ref{prop:seshadribound}. 
However, Lazarsfeld~\cite{Lazarsfeld:seshadri} combined the work of 
McDuff and Polterovich~\cite{MP} with work of Buser and Sarnak 
on minimal period lengths to deduce that there exist principally polarized 
abelian varieties~$(A^n,\Theta)$ of complex dimension~$n$ with
$$
\eps(A^n,\Theta) \,\geq\, \tfrac 1 4 (2n!)^{\frac 1 n}.
$$
Bauer has generalized this, showing that there exist 
polarized abelian varieties~$(A^n,L)$ of arbitrary type $(d_1,\dots,d_n)$ with
\begin{equation} \label{eq:seshadri4}
\eps(A^n,L) \,\geq\, \tfrac 1 4 (2 n!d_1\dots d_n)^{\frac 1 n}.
\end{equation}
While this is only a factor of less than $4$ away from the upper bound 
of~\eqref{eq:seshadri2}, the volume fraction filled by the symplectic ball 
predicted from this lower bound is $2\left(\frac 1 4\right)^n$, 
and hence tends exponentially to zero as~$n \to \infty$.

\section{Proof of Theorem~\ref{t:2}} \label{s:pack}

\subsection{Irrational case}\label{ss:irrat}           
We will use the following result of Buchdahl~\cite{Buchdahl} and Lamari~\cite{Lamari}:
\begin{thm}[\cite{Buchdahl, Lamari}]
Let $(X,J)$ be a compact complex surface. 
A cohomology class $\alpha \in H^{1,1}(X;\R)$ admits a K\"ahler representative
compatible with the complex structure~$J$ 
if $\alpha \cup \alpha >0$, $\alpha\cup [\rho] >0$ for some positive closed $(1,1)$-form~$\rho$
on~$X$ and $\alpha \cdot [D]>0$ for every effective divisor $D \subset X$.
\footnote{Notice that the form~$\rho$ is a K\"ahler form. Therefore, Buchdahl's condition that the first Betti number of~$X$ is even is automatic.}
\end{thm}
In symplectic language, 
the last condition means that the class~$\alpha$ should integrate positively over every compact holomorphic curve in~$X$. 

Our argument is based on the following result, whose proof is deferred until the next subsection.

\begin{prop}  \label{prop:no}
Any irrational linear symplectic form $\om$ on $T^4$ may be identified with a 
K\"ahler form on a torus  $T= \C^2/\Lambda$ that has no 
nonconstant compact holomorphic curves.
\end{prop}

\begin{prop}  \label{prop:t2.irrational}
Theorem~\ref{t:2}  holds for irrational tori.
\end{prop}

\begin{proof}  
We must show that any disjoint union $\coprod_{j=1}^k \ov B\,\!^4(a_j)$ 
of balls symplectically embeds into~$(T^4,\om)$ provided only that 
the volume constraint is satisfied.  
By Proposition~\ref{prop:no} there is a symplectomorphism 
from $(T^4,\om)$ to~$(T,\om_J)$ 
where $(T,J,\om_J)$ is the K\"ahler torus found in Proposition~\ref{prop:no}.
Let $(T_k,J_k)$ be the complex blow-up of~$T$ at~$k$ generic points~$x_j$, 
and 
consider the cohomology class 
$\al := \pi^*([\omega_J]) - \sum_{j=1}^k a_j \PD [\Sigma_j]$, where
$\pi \colon T_k\to T$ is the blow-down map and the $\Sigma_j$ are the exceptional divisors.
Since the complex structure~$J$ and the symplectic form~$\omega_J$ on~$T$ have constant coefficients, 
we find for each $j$ an $\varepsilon_j >0$ and an embedding $\psi^j \colon B^4(\varepsilon_j) \to T$
with $\psi^j(0) = x_j$ that is holomorphic and symplectic
(i.e., $\psi^j_* J_{\can} = J$ and $\psi^j_* \omega_{\can} = \omega$),
see \cite[Exercise~2.52\,(iii)]{MS}.
Take $\varepsilon >0$ such that $\varepsilon \le \varepsilon_j$ for each $j=1,\dots,k$ and such that
$\int_{T} \omega_J \wedge \omega_J > \sum_{j=1}^k \varepsilon a_j$.
Let $\rho$ be the K\"ahler form on $T_k$ corresponding to the blow-up defined by the $k$ embeddings
$\psi^j \colon B^4(\varepsilon) \to T$. 
Then $[\rho] = \pi^*([\omega_J]) - \sum_{j=1}^k \varepsilon \PD [\Sigma_j]$,
whence $\alpha \cup [\rho] >0$.
Furthermore, the volume condition gives~$\al^2>0$,
and the only compact holomorphic curves in~$(T_k,J_k)$ are the~$\Sigma_j$.  
The criterion of Buchdahl and Lamari thus holds for the class~$\al$.
Therefore there is a K\"ahler form $\tau$ on~$(T_k,J_k)$ in class~$\al$. 
Blowing down the form~$\tau$ we obtain a symplectic form~$\eta$ on~$T$ 
in class~$[\omega_J]$, and disjoint symplectically embedded balls 
$\ov B\,\!^4(a_1), \dots, \ov B\,\!^4(a_k)$ in $(T,\eta)$.
It remains to show that $\eta$ is isotopic to $\omega_J$.
Then Moser's argument shows that $(T,\eta)$ and $(T,\omega_J)$ 
are symplectomorphic. 
Hence the balls $\coprod_j \ov B\,\!^4(a_j)$ also symplectically embed into~$(T,\om_J)$.

To find an isotopy from $\eta$ to $\omega_J$,
we consider for each $s \in (0,1]$
the cohomology class 
$\al_s := \pi^*([\omega_J]) - \sum_{j=1}^k s a_j \PD [\Sigma_j]$,
and repeat the above construction to find K\"ahler forms~$\tau_s$ on $(T_k,J_k)$
in class~$\alpha_s$, $s \in (0,1]$. 
Moreover, choose $s_0>0$ such that $s_0 a_j < \varepsilon$ for all~$j$,
and for $s \in (0,s_0)$ let $\rho_s$ be the K\"ahler form on $T_k$ corresponding to the blow-up defined by the $k$ embeddings
$\psi^j \colon B^4(s_0 a_j) \to T$. 
Denote by $\Omega (J_k,\{\Sigma_j\})$ the space of $J_k$-tame symplectic forms
on~$T_k$ that restrict to 
symplectic forms on the~$\Sigma_j$.
Then each form~$\tau_s$, $s \in (0,1]$, 
and each form $\rho_s$, $s \in (0,s_0)$, 
belongs to~$\Omega (J_k,\{\Sigma_j\})$.
Since the space of forms in $\Omega (J_k,\{\Sigma_j\})$ 
in a given cohomology class is convex, 
we can alter the collection $\{\tau_s\}$ to a smooth family $\{\tau_s'\}$
of forms in $\Omega (J_k,\{\Sigma_j\})$ such that $\tau_s'$ 
is cohomologous to~$\tau_s$, such that $\tau_s' = \rho_s$ 
for $s \in (0,s_0/2)$, and such that $\tau_1'=\tau_1$.
Now blow down the forms $\tau_s'$ to obtain a smooth family~$\eta_s$
of symplectic forms on~$T$ in class $[\om_J]$.
By construction, $\eta_1 = \eta$ and $\eta_s = \omega_J$ for $s \in (0,s_0/2)$,
as required.
\end{proof}

\begin{remark} \label{rem:K3}
{\rm
The only property of irrational linear symplectic forms on~$T^4$ that
we used here was the existence of a compatible complex structure~$J$
such that $H^2(T;\ZZ) \cap H^{1,1}(T;\CC)$ vanishes, i.e.,
such that the Picard number of $(T,J)$ vanishes.
According to the Enriques--Kodaira classification, there is exactly
one other class of compact K\"ahler surfaces with this property,
namely K3-surfaces with Picard number~$0$.
Repeating the above proof, we find that the conclusions of
Theorem~\ref{t:2} also hold true for all 
K\"ahler forms on K3-surfaces for which a compatible complex structure has 
vanishing Picard number.
}
\end{remark}

\subsection{Irrational tori with no curves}\label{ss:no}

We now prove Proposition \ref{prop:no}.

We begin by finding complex tori with no nonconstant compact holomorphic curves.
Consider $\C^2=\C e_1 \oplus \C e_2$ and denote by $e_1$, $e_2$, $e_3=\sqrt{-1}\,e_1$, $e_4=\sqrt{-1}\,e_2$ the standard real basis.
Choose real numbers $p,q,r,s$ such that 
\begin{gather}\label{eq:LaP}
p,\, q,\, r,\, s \,\mbox{ are rationally independent, and } 
ps-qr \, \mbox{ is irrational}. 
\end{gather}
Consider the quotient of $\C^2$ by the lattice $\Lambda_P$ spanned by
$$
\lambda_1 = e_1, \;\;
\lambda_2 = e_2, \;\;
\lambda_3 = p \,e_3+ q \,e_4, \;\;
\lambda_4 = r \,e_3 + s \,e_4.
$$

The following result is 
extracted from the Appendix in~\cite{ElencwajgForster}.  
We repeat it here for the sake of completeness.

\begin{lemma} \label{l:no}  
Under the assumption \eqref{eq:LaP},
the torus $T=\C^2/\La_P$ contains no nonconstant compact holomorphic curves. 
\end{lemma}

\begin{proof}
A nonconstant compact holomorphic curve 
would represent a nonzero class in $H_{1,1}(T;\C) \cap H_2(T;\Z)$.
By duality it thus suffices to prove that $H^2(T;\Z) \cap H^{1,1}(T;\C)=\{0\}$. 
Write the complex coordinates of $\C^2$ as $z_j=x_j+\sqrt{-1}\,y_j$.
Since every class in $H^2(T;\C)$ can be represented by a form with constant coefficients,
every class in $H^{1,1}(T;\C)$ has a representative of the form
%
\begin{equation}\label{eq:om}
\om = x\, dx_1\wedge dy_1 + y\, dx_2\wedge dy_2 + u(dx_1\wedge dy_2+dx_2\wedge dy_1) + v(dx_1\wedge dx_2+dy_1\wedge dy_2),
\end{equation}
where the coefficients $x,y,u,v$ are complex constants. The class $[\om]$ will be integral if and only if the coefficients of the matrix
\begin{equation}\label{eq:matrix}
\om(\lambda_i,\lambda_j) \,=\, 
\begin{pmatrix}
0 & v & px+qu & rx+su \\
-v& 0 & pu+qy & ru+sy \\
-(px+qu) & -(pu+qy) & 0 & v(ps-qr) \\
-(rx+su) & -(ru+sy) & -v(ps-qr) & 0
\end{pmatrix} ,
\end{equation}
are integers, i.e.\/ if
 $$
\begin{array}{ll}
{\rm (i)}&\quad v,\, v(ps-qr) \in \ZZ; \\
{\rm (ii)}&\quad px+qu,\; r x+su, \; pu+qy,\; r u+sy \in \ZZ .
\end{array}
$$
Since we have chosen $ps-qr$ irrational, the two conditions~(i) imply $v=0$.
Assume that $x,y,u$ fulfill the conditions~(ii).
We then find $n_1,n_2,n_3,n_4 \in \ZZ$ with
$$
\begin{array} {llllllllll}
px+qu       &=& n_1,   &&&&& pu+qy      &=& n_3,  \\
r x + su &=& n_2,   &&&&& r u +sy &=& n_4 .
\end{array}
$$
We can eliminate $x$ and $y$ from the above equations and obtain
\begin{eqnarray*}
(ps-qr)u &=& -n_1 r + n_2 p \\
(ps-qr)u &=& \phantom{-}n_3s-n_4q,
\end{eqnarray*}
which implies $-n_1r + n_2p - n_3s + n_4q =0$.
Since we have chosen $p, q, r, s$ to be rationally independent, 
it follows that $n_1,n_2,n_3,n_4$ must vanish. 
Hence $u$ and therefore also $x$ and $y$ vanish. We conclude that $\om=0$, as we wanted to show.
\end{proof}


Now we start with a torus $T^4= \R^4/\Lambda$ with a linear symplectic form~$\om$ 
representing an irrational cohomology class. 
Given an integral basis $\lambda_1,\dots,\lambda_4$ for~$\Lambda$, 
the symplectic form~$\omega$ can be represented by a matrix $B=(b_{ij})$,   
where $b_{ij}=\om(\lambda_i,\lambda_j)$.
We denote by $\lambda_1^*, \dots, \lambda_4^*$ the basis dual to $\lambda_1,\dots,\lambda_4$, 
and we may assume that the ordering has been chosen such that $\om \wedge \om$ is a positive multiple 
of $\lambda_1^* \wedge \lambda_3^* \wedge \lambda_2^* \wedge \lambda_4^*$.

\begin{lemma}\label{l:om}   
In the situation just described, after changing the basis of~$\R^4$ by an element of~$\SL(4,\Z)$, 
we may represent~$\om$ by a matrix~$B'$ where
\begin{itemize}
\item[(i)]
 the entries $b'_{12}$, $b'_{34}$ either both vanish or they are  rationally independent and positive, 
 and 
\item[(ii)]  
the vector $(b'_{13}, b'_{14}, b'_{23}, b'_{24})$ is not a multiple of a rational vector.
\end{itemize}
\end{lemma}
\begin{proof}    
Suppose first that there is a permutation $i_1,\dots,i_4$ of $\{1,\dots ,4\}$ so that
$b_{i_1i_2} = b_{i_3i_4}=0$. Then we can change basis (preserving orientation)
so that $b_{12}' = b_{34}'=0$. 
Condition~(ii) is then automatic since $\om$ is irrational. 

In all other cases
we can permute the basis (preserving orientation) so that $b_{12} \ne 0$, $b_{34} \ne 0$, 
and so that at least two of
the elements $b_{12}, b_{34}, b_{1j}$ are rationally independent, 
where $j=3$ or~$4$.  
If condition~(i) is not satisfied, 
we change basis by replacing  $\lambda_2^*$ by $\lambda_2^*+k\lambda_j^*$
and leaving the other elements fixed.
Then $b_{12}$ changes to $b_{12}'=b_{12} + kb_{1j}$ and $b_{34}'=b_{34}$, so that 
for suitable $k\in \Z$ we may assume that $b_{12}'$, $b_{34}'$ are rationally independent 
and of the same sign. 
If they are both negative, we can change their signs by 
interchanging $\lambda_1$ and $\lambda_2$ and interchanging $\lambda_3$ and $\lambda_4$.  
This achieves~(i).

If (ii) does not hold, we may assume that
the vector $\beta=(b_{13}, b_{14}, b_{23}, b_{24})$ does not vanish,
because otherwise by a permutation we could have arranged the situation with $b_{12}'=b_{34}'=0$ 
from the beginning of the proof, and as observed there (ii) is automatic in that case.

So one of the entries of $\beta$ must be nonzero and hence rationally independent of either  
$b_{12}$ or $b_{34}$. We will consider the case that $b_{13}$ and $b_{12}$ are rationally independent, 
the other cases being treated in a similar fashion. 
Now we change basis, replacing  $\lambda_4^*$ by $\lambda_4^*+k\lambda_1^*$  
and leaving the other elements fixed.  
Then $b_{ij}' = b_{ij}$ if $i,j\ne 4$, while $b_{i4}'=b_{i4}-kb_{1i}$.
In particular,  
$$
b_{12}'=b_{12}, \qquad  b_{34}' = b_{34} - k b_{13}, \qquad b_{13}'=b_{13}, \qquad
b_{24}'= b_{24}-kb_{12}.
$$
Hence (ii) holds if $k\ne 0$, since $b_{13}', b_{24}'$ are rationally independent.
Further (i) will hold if we choose $k$ so that $- k b_{13}>0$.

The proof in the other cases is similar.  
In particular if $b_{13}=0$ but $b_{14}\ne 0$ we use a base change that alters $\lambda_3^*$ 
instead of~$\lambda_4^*$.  
\end{proof}
 
\begin{remark} \label{rem:det}
{\rm 
Note that 
$$
\om \wedge \om \,=\, (b_{13}b_{24}-b_{14}b_{23}-b_{12}b_{34})\, \lambda_1^*\wedge\lambda_3^*\wedge\lambda_2^*\wedge\lambda_4^*.
$$
Since the base change was orientation preserving, the coefficient is still positive, 
and so in the new basis for~$\Lambda$ constructed in Lemma~\ref{l:om} we necessarily have $b_{13}'b_{24}'-b_{14}'b_{23}'>0$, since $-b_{12}'b_{34}'\le 0$ by~(i).
}
\end{remark}

\begin{proof}[Proof of Proposition \ref{prop:no}.]
We are given a torus $T^4=\R^4/\Lambda$ with a linear irrational symplectic form. 
We assume that we have chosen a basis $\lambda_1,\dots,\lambda_4$ for~$\Lambda$ 
such that the matrix~$B$ determined from $\om(\la_i,\la_j) = b_{ij}$ satisfies 
the conditions stated in Lemma~\ref{l:om}. 
Our goal is to identify~$\Lambda$ with a suitable period lattice~$\Lambda_P$ 
of the form discussed in Lemma~\ref{l:no}, where the coefficients $p,q,r,s$ 
are still to be determined. 
Moreover, we want that under this identification the form~$\om$ is as in~\eqref{eq:om} for suitable constants $x,y,u,v\in \R$, and is compatible with the standard complex structure~$J_0$ on~$\C^2$.    

With respect to the real standard basis $e_1,e_2,e_3=\sqrt{-1}\,e_1,e_4= \sqrt{-1}\,e_2$ of $\C^2$ the symmetric bilinear form~$g$ associated to a symplectic form $\om$ as in \eqref{eq:om} and the standard complex structure $J_0$ is represented by the matrix
$$
(g_{ij}) \,=\,\om(e_i,J_0e_j) \,=\,
\left(\begin{array}{cccc}
x & u & 0 & -v\\
u & y & v & 0\\
0 & v & x & u\\
-v & 0 & u & y
\end{array}\right).
$$
The compatibility of $\om$ with $J_0$ requires this matrix to be positive definite, and this holds if and only if all leading principal minors are positive. This will be the case if and only if
\begin{equation}\label{eq:positive}
x>0 \quad \mbox{ and } \quad xy-u^2-v^2>0,
\end{equation}
since the other two conditions $xy-u^2>0$ and $(xy-u^2-v^2)^2$ 
are then necessarily also satisfied.

To construct the lattice $\Lambda_P$ and find the coefficients of $\om$, a comparison with equation~\eqref{eq:matrix} of the proof 
of Lemma~\ref{l:no} shows that we want to solve the equations
$$
\begin{array}{llll}
b_{12} &= \;v, & \qquad
b_{13}  &= \; px + qu, \\
b_{14} &= \; rx +su, & \qquad
b_{23} &= \; pu+qy, \\
b_{24} &= \; ru+sy, & \qquad
b_{34} &= \; v(ps-qr).
\end{array}
$$
The middle four equations can be rewritten as
\begin{equation}\label{eq:mat}
\left(\begin{array}{ccc} 
q & p & 0 \\ s & r & 0\\ p & 0 & q\\ r & 0 & s
\end{array}\right) 
\left(\begin{array}{c} u\\x\\y\end{array}\right)
=
\left(\begin{array}{c} b_{13}\\ b_{14}\\ b_{23}\\ b_{24} \end{array}\right).
\end{equation}
Here the vector on the right hand side is given and nonzero. 
For fixed $p$, $q$, $r$, $s$ 
with $ps-qr \neq 0$
this overdetermined system of equations will have 
a solution $(u,x,y)$ if the compatibility condition
\begin{equation}\label{eq:comp}
r\, b_{13} -p\, b_{14} = q\, b_{24} - s\, b_{23}
\end{equation}
is satisfied.

\MS 
\begin{lemma}\label{l:choose}
If the vector $(b_{13}, b_{14}, b_{23} , b_{24})$ is not a multiple of a rational vector, 
there exists a solution $(p',q',r',s') \in \R^4$ of~\eqref{eq:comp} in rationally independent 
real numbers satisfying $s'b_{13}-q'b_{14} >0$ and $D := p's'-q'r'>0$. 
\end{lemma}

\begin{proof}
The inequalities $sb_{13}-qb_{14}>0$ and $D>0$ define an open set in $\R^4$, 
which we denote by~$\mathcal O$.
Similarly, for given $b_{ij}$ the equation~\eqref{eq:comp} defines a hyperplane~$H$ in~$\R^4$. 
The intersection $\mathcal O \cap H$ is not empty, since in view of Remark~\ref{rem:det} 
the point with coordinates $p=b_{13}$, $q=b_{23}$, $r=b_{14}$ and $s=b_{24}$ belongs to it. 

On the other hand, a point $(p,q,r,s)$ has rationally dependent coordinates if and only if 
it solves some equation $n_1p+n_2q+n_3r+n_4s=0$ with integral coefficients~$n_i$. 
Since its defining vector is not a (multiple of a) rational vector, 
$H$ is transverse to this 
countable set of hyperplanes, and so there is some point in the open subset $\mathcal O \cap H$ 
that does not lie on any of these hyperplanes. 
This point has the desired properties.
\end{proof}

Given $(p',q',r',s')$ as in the lemma, the solution of the matrix equation~\eqref{eq:mat} is
\begin{eqnarray*}
x'=\frac 1 {D}(s'b_{13}-q'b_{14}), \qquad 
y'=\frac 1 {D}(p'b_{24}-r'b_{23}),\\
u'=\frac 1 {D}(p'b_{14}-r'b_{13})= \frac 1 {D}(s'b_{23}-q'b_{24}).
\end{eqnarray*}

If $b_{12}=b_{34}=0$, the final two equations 
$$
b_{12}=v, \qquad b_{34}= v(p's'-q'r'),
$$
have the trivial solution $v=0$. 
In this case choose $\rho >0$ such that $\rho^2 (p'q'-r's')$ is irrational,
and define $(p,q,r,s) := \rho (p',q',r',s')$ and $(u,x,y) := \rho^{-1} (u',x',y')$.

If $b_{12}, b_{34}$ are rationally independent, 
we need to rescale the above solution so that $ps-qr=\frac {b_{34}}{b_{12}}$. 
Therefore, define 
$(p,q,r,s) := \rho (p',q',r',s')$ and $(u,x,y) := \rho^{-1} (u',x',y')$,
where $\rho := \sqrt{\frac {b_{34}}{b_{12} D}}$. 
Notice that $ps-qr = \frac {b_{34}}{b_{12}}$ is 
automatically
irrational by part~(i) of Lemma~\ref{l:om} in this case.
Now, choosing $v=b_{12}$, all six equations are satisfied.
 
With this construction, we have found a lattice $\Lambda_P$ 
and coefficients $x,y,u,v\in \R$ such that,
after identifying $\Lambda$ with $\Lambda_P$
by mapping the basis vectors~$\lambda_j$ of~$\Lambda$ to the basis vectors~$\lambda_j$ of~$\Lambda_P$,
the symplectic form $\om$ is as in \eqref{eq:om}. 
To check the positivity condition~\eqref{eq:positive}, 
note that $x>0$ by construction, and a computation shows that
$xy-u^2-v^2$ equals a positive multiple of $b_{13}b_{24}-b_{14}b_{23}-b_{12}b_{34}$, 
which was observed to be positive in Remark~\ref{rem:det}. 
In summary, we have shown that $\om$ is compatible with the standard complex structure~$J_0$ on $\C^2$. 
Finally, the lattice $\Lambda_P$ by construction satisfies the assumptions of Lemma~\ref{l:no}, and so we have proven the proposition.
\end{proof}


\subsection{Product tori} \label{s:muge2}

In this section, we complete the proof of Theorem~\ref{t:2}  
by treating the case of product tori~$T(1,\mu)$.
We can assume that $\mu \ge 1$.
If $\mu$ is irrational, Theorem~\ref{t:2} has been proven in 
Sections~\ref{ss:irrat} and~\ref{ss:no}.
If $\mu = 1$ or $\mu \ge 2$ is rational,
Theorem~\ref{t:2} follows from 
Proposition~\ref{p:prod.balls} below
(notice that for $\mu \ge 2$ the condition on the individual~$a_j$ and~$b_j$ 
appearing in this proposition are automatic from volume considerations). 
Finally, if $\mu \in (1,2)$ is rational, 
Theorem~\ref{t:2} follows from the case of rational~$\mu \ge 2$ as in the proof of Corollary~\ref{c:mu}.

\begin{proposition} \label{p:prod.balls}
Let $\mu \ge 1$ be rational and 
let $\ov B(a_1), \dots,\ov B(a_k)$ be a collection of balls such that
$$
\Vol \Bigl( \coprod_{j=1}^k \ov B(a_j) \Bigr) \,<\, \mu 
$$
and such that $a_j < \mu$ for all~$j$.
Then there exists a symplectic embedding of $\coprod_{j=1}^k \ov B(a_j) $ 
into the open disc bundle $T^2(1) \times D^2(\mu)$.
\end{proposition}

\proof
Denote by $S^2(\mu)$ the $2$-sphere endowed with an area form of area~$\mu$.
Biran has shown in \cite[Proof of Corollary~5.C]{B1} that 
$\coprod_{j=1}^k \ov B(a_j)$ 
symplectically embeds into $T^2(1) \times S^2(\mu)$ whenever 
$\Vol \Bigl( \coprod_{j=1}^k \ov B(a_j) \Bigr) < \mu$ 
and $a_j < \mu$ for all~$j$.
We will arrange this embedding in such a way that the balls lie in 
$T^2(1) \times (S^2(\mu) \less z_0) = T^2(1) \times D^2(\mu)$, where $z_0 \in S^2$.
Such a construction has been carried out in~\cite{B2} in a slightly different situation.
We shall outline the construction, pointing out the difference.

A symplectic embedding of 
$\coprod_{j=1}^k \ov B(a_j) $ into $T^2(1) \times (S^2(\mu) \less z_0)$ 
is obtained by constructing
a smooth family of cohomologous forms 
$\om_s$, $s \in [0,1]$, on $T^2 \times S^2$ with the following properties:
\begin{itemize}
\item $\om_0$ is the product form on $T^2(1)\times S^2(\mu)$;
\item each $\om_s$ is nondegenerate on the torus $Z : = T^2 \times \{z_0\}$;
\item for each $s\in [0,1]$ there is a symplectic embedding of 
$\coprod_{j=1}^k s \ov B(a_j)$ into $\bigl( (T^2 \times S^2) \less Z , \omega_s \bigr)$. 
\end{itemize}
For then a standard Moser argument shows that there is a family of diffeomorphisms 
$\psi_s \colon (T^2\times S^2, Z)\to (T^2\times S^2, Z)$ 
such that $\psi_1^* \om_1= \om_0$.
Therefore $\coprod_{j=1}^k  \ov B(a_j)$  symplectically embeds into 
$\bigl( T^2\times (S^2 \less z_0), \om_0 \bigr) = T^2(1)\times D^2(\mu)$.

The family $\om_s$ is constructed in much the same way as in the proof of 
Theorem~\ref{t:2} for irrational tori. 
In other words, the problem is converted into one of constructing 
suitable forms~$\tau_s$ on the $k$-fold blow-up.  
The only difference is that we can no longer find the required forms~$\tau_s$ on the blow-up via the Buchdahl--Lamari criterion; 
instead we must use symplectic inflation as in~\cite{LM.class, M-isotopy, B1}. 
In order for $Z$ to be $\tau_s$-symplectic, it suffices to work only with almost complex structures~$J$ 
for which $Z$ is $J$-holomorphic. 

More precisely, let $(M,\omega_0) := T^2(1) \times S^2(\mu)$, 
and choose different points $x_1, \dots, x_k$ in~$M \less Z$.
Let $J_0$ be the standard product complex structure on~$M$, 
and choose neighborhoods $U(Z) = T^2 \times D(z_0)$, 
$B_\varepsilon(x_1), \dots, B_\varepsilon(x_k)$ with disjoint closures.
At each point $x_j$ form the K\"ahler blow-up of size~$\varepsilon$.
Denote the resulting K\"ahler manifold by 
$(M_k, J_k, \tau_\varepsilon)$.
Let $\Sigma_1, \dots, \Sigma_k$ be the exceptional divisors, and let $E_j$ be the 
homology class of~$\Sigma_j$.
Note that $\pi \colon M_k \less \bigcup_j \Sigma_j \to M \less \bigcup_j x_j$ 
is a diffeomorphism, 
and denote $\pi^{-1}(Z)$ and $\pi^{-1}(U(Z))$ by $\widetilde Z$ and $U(\widetilde Z)$.
Choose mutually disjoint neighborhoods $U(\Sigma_j)$ that are also 
disjoint from~$\widetilde Z$.
Let $\cj$ be the space of $\tau_\varepsilon$-tame almost complex structures on~$M_k$, 
and denote by $\cj'$ the subspace of those~$J$ in~$\cj$ that restrict to~$J_k$ on 
$U(\widetilde Z) \cup U(\Sigma_1) \cup \dots \cup U(\Sigma_k)$.
As in~\cite{B1, B2},
the class $[\omega_0] \in H^2 (T^2 \times S^2;\RR)$ is rational by assumption.
Choose rational numbers $a_j' \in (a_j, \mu)$ such that
$\Vol \Bigl( \coprod_{j=1}^k \ov B(a_j) \Bigr) < \mu$.
We then find
$n \in \NN$ such that 
\begin{equation} \label{e:inflationclass}
A \,:=\, n \biggl( \PD [\pi^* \omega_0] - \sum_{j=1}^k a_j' E_j \biggr)
\end{equation}
belongs to $H_2(M_k;\ZZ)$.
In order to construct the forms~$\tau_s$ we wish to inflate the form~$\tau_\varepsilon$ 
along the class~$A$.
To this end we need to represent~$A$ by a smooth connected embedded reduced 
$J$-holomorphic curve~$C$ for some $J \in \cj'$.
To find such a curve~$C$ one can closely follow the proof of 
Lemma~2.2.B in~\cite{B2}, where, however, 
the curve~$\widetilde Z$ has negative self-intersection.
This causes no problem if one observes that, by the maximum principle, 
for every $J' \in \cj'$, 
every $J'$-holomorphic curve that is entirely contained in $U(\widetilde Z)$ 
must be (a multiple cover of) a torus $\pi^{-1} (T^2 \times \{z\})$ 
with $z \in D(z_0)$.
For $B \in H_2(M_k;\ZZ)$, 
the maximal number of generic points through which, for generic $J \in \cj$,
a $J$-holomorphic curve can pass, is
$k(B) := \frac 12 \bigl( B \cdot B + c_1(B) \bigr)$. 
Let 
\begin{equation} \label{e:cusp}
A \,=\, \sum_j m_j A_j + m [\widetilde Z], \quad m \ge 1,
\end{equation}
be an $A$-cusp configuration with $m \ge 1$.
Denote by $k_{\text{cusp}}(A)$ the maximal number of generic points 
through which, for a generic subset of $J'$ in~$\cj'$,
a $J'$-cusp-curve with configuration~\eqref{e:cusp} can pass.
As in \cite[pp.~148--151]{B2},
the last step in the existence proof for the curve~$C$ 
is to understand that $k_{\text{cusp}}(A) < k(A)$.
To see this, 
consider the class $\bar A := A - m[\widetilde Z]$.
In view of~\eqref{e:inflationclass}, the class~$A$ is not a multiple of $[\widetilde Z]$, 
and hence the class~$\bar A$ is non-trivial.
Biran showed in the proof of Lemma~2.2.B of~\cite{B2} that for generic $J' \in \cj'$,
no (cusp-)curve in class~$\bar A$ can pass through more than $k(\bar A)$ generic points.
Since in the definition of $k_{\text{cusp}}(A)$ we may consider points 
outside~$U(\widetilde Z)$ only, it follows that $k_{\text{cusp}}(A) \le k(\bar A)$.
Moreover, 
using $c_1(\widetilde Z)=0$, $[\widetilde Z] \cdot [\widetilde Z] =0$
and $A \cdot [\widetilde Z] = n \int_Z \omega_0 >0$, $m \ge 1$,
we compute
$$
k(\bar A) \,=\, k(A) - m \bar A \cdot [\widetilde Z] 
          \,=\, k(A) - m A \cdot [\widetilde Z] \,<\, k(A) .
$$
Altogether, $k_{\text{cusp}}(A) \le k(\bar A) < k(A)$.
\proofend


\section{Basic symplectic mappings} \label{s:basic}

\ni
In this section we describe an elementary symplectic embedding construction.
It will be applied in Section~\ref{s:t1} to prove that $p\bigl(T(1,1)\bigr)=1$.
We write $\RR^2(\bx) := \RR^2(x_1,x_2)$ and $\RR^2(\by) := \RR^2(y_1,y_2)$.

\subsection{Diamonds} 

Consider the ``diamond'' of size $a$
$$
\Diamond (a) \,:=\, \left\{ (x_1,x_2) \in \RR^2(\bx) \mid |x_1|+|x_2| < \tfrac a2 \right\}
\;\subset\; \RR^2(\bx) ,
$$
see Figure~\ref{fig.dist}~(I).  

\begin{lemma} \label{l:alpha} 
For each $\gve >0$ the ball $B^4(a)$ symplectically embeds into
$\Diamond (a+\gve) \times (0,1)^2 \subset \RR^2(\bx) \times \RR^2(\by)$.
\end{lemma}

\proof
Let $D(a) \subset \RR^2(z) = \RR^2(x,y)$ be the open disc of area~$a$.
Choose an area and orientation preserving embedding
$$
\sigma \colon D(a) \,\to\, \left( -\frac{a+\gve}{2}, \frac{a+\gve}{2} \right) \times (0,1)
$$ 
such that
\begin{equation} \label{ine:xz}
\bigl| x(\sigma (z))\bigr| \,<\, \frac 12 \pi |z|^2 + \frac \gve 2 \quad \mbox{ for all }\, z \in D(a).
\end{equation}
Figure~\ref{fig.diam} shows such an embedding.
For details we refer to Lemma~3.1.8 of~\cite{Sch-book}.

%
\begin{figure}[h]
 \begin{center}
  \psfrag{x}{$x$}
  \psfrag{y}{$y$}
  \psfrag{z}{$z$}
  \psfrag{1}{$1$}
  \psfrag{ga}{$\sigma$}
  \psfrag{a-}{$-\frac{a+\gve}2$}
  \psfrag{a+}{$\frac{a+\gve}2$}
  \leavevmode\epsfbox{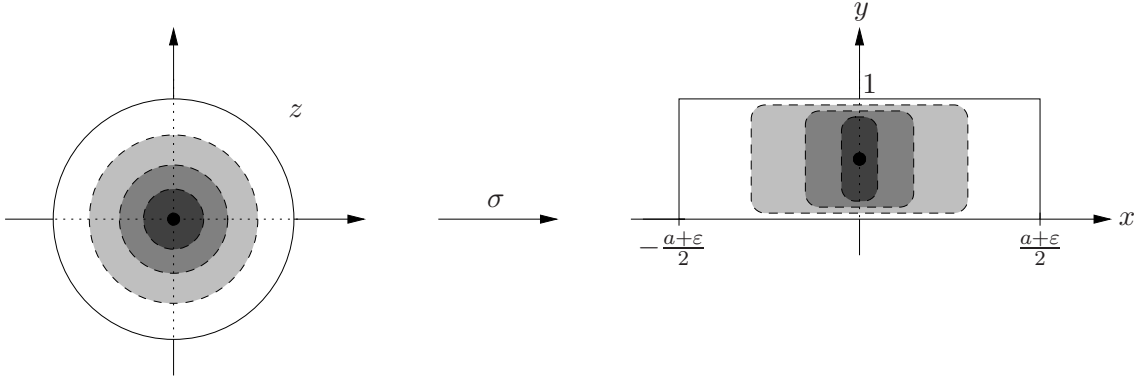}
 \end{center}
 \caption{The map $\sigma$.}
 \label{fig.diam}
\end{figure}

\m
\ni
We claim that the symplectic embedding $\sigma \times \sigma \colon D(a) \times D(a) \to \RR^4$
maps $B^4(a)$ to $\Diamond (a+\varepsilon) \times (0,1)^2$.
Indeed, for $(z_1,z_2) \in B^4(a)$ we have $\pi (|z_1|^2+|z_2|^2) < a$.
Together with~\eqref{ine:xz} we can estimate
\begin{eqnarray*}
\bigl| x_1 ((\sigma \times \sigma) (z_1,z_2) )\bigr| + \bigl| x_2 ((\sigma \times \sigma) (z_1,z_2) )\bigr|
&=& 
\bigl| x_1 (\sigma (z_1) )\bigr| + \bigl| x_2 (\sigma (z_2) )\bigr| \\
&<&
\frac 12 \left( \pi |z_1|^2 + \pi |z_2|^2 \right) + \gve \\ 
&<& \frac a2 + \gve ,
\end{eqnarray*}
as claimed.
\proofend

\begin{cor}\label{cor:alpha}
The open ball $B^4(a)$ is symplectomorphic to
$\Diamond (a) \times (0,1)^2 \subset \RR^2(\bx) \times \RR^2(\by)$.
\end{cor}
\begin{proof} This follows by combining the above lemma with Lemma~\ref{l:nest} below.
\end{proof}

\begin{lemma} \label{l:nest}
Let $V \subset \RR^4$ be a bounded domain such that for each compact subset
$K \subset V$
there exists $\hat a < a$ and a symplectic embedding 
$\hat \varphi \colon B^4(\hat a) \to V$
such that $\Ima \hat\varphi \supset K$.
Then $V$ is symplectomorphic to~$B^4(a)$.
\end{lemma}

\proof
Choose a sequence $K_1 \subset K_2 \subset \dots$ of compact subsets of~$V$
such that $\bigcup_j K_j = V$.
Using the assumption of the lemma and the result from~\cite{M-blowup} that
the space of symplectic embeddings of a closed ball into an open ball is connected,
we construct a sequence $a_1'<a_1 < a_2'<a_2 < \dots$ with $a_j \to a$
and a sequence of symplectic embeddings $\varphi_j \colon B^4(a_j) \to V$
such that  
$\varphi_j (B^4(a_j')) \supset K_j$ and $\varphi_{j+1} |_{B^4(a_j')} = \varphi_j |_{B^4(a_j')}$.

Define $\varphi \colon B^4(a) \to V$ by $\varphi (z) = \varphi_j(z)$ 
if $z \in B^4(a_j')$.
Then $\varphi$ is a well-defined symplectic embedding.
Moreover, $\varphi (B^4(a_j')) = \varphi_j (B^4(a_j')) \supset K_j$.
Hence $\varphi$ is onto $\bigcup_j K_j = V$.
\proofend

\subsection{Distorted diamonds}
All of our embeddings, besides one, will start from a diamond~$\Diamond (a)$.
For our full filling of $T(1,1)$, however, 
we shall need to start from a {\it distorted diamond}. 
 
Fix $a>0$.
Let $u_+ \colon [0,a] \to \RR$ be a continuous,
nondecreasing function that for convenience we take to
be piecewise-linear. Suppose further that
$$
u_+(0)=0 \quad \mbox{ and } \quad
u_+'(\rho) \in [0,1] \;\mbox{ on the linear pieces}.
$$
Define the piecewise-linear function $u_- \colon [0,a] \to \RR$ by $u_-(\rho) = u_+(\rho)-\rho$. 
Then $u_-(0)=0$ and $u_-'(\rho) \in [-1,0]$,
so that $u_-$ is nonincreasing. 
Moreover, $u_+'(\rho)-u_-'(\rho)=1$ on the linear pieces,
and $u_+(a)-u_-(a)=a$.
Let
$$
\sigma_u \colon D(a) \,\to\, 
\left( u_-(a) - \frac \varepsilon 2,\ u_+(a) + \frac \varepsilon 2 \right) \times (0,1)
$$ 
be a symplectic embedding such that
\begin{equation*} 
x(\sigma_u (z)) \,\in\, \left( u_-(\rho)- \frac \gve 2,\ u_+(\rho)+ \frac \gve 2 \right)
\quad \mbox{ for all}\, z \in D(a) \mbox{ with } \pi |z|^2 < \rho .
\end{equation*}
Figure~\ref{fig.diam3} shows the image of concentric circles of 
the map~$\sigma_u$ for $a=3$ and for the function $u_+ \colon[0,3] \to \RR$
with slope $\frac 12$ on $[0,\frac 32]$ and slope $\frac 13$ on $[\frac 32,3]$. 
%

%
\begin{figure}[h]
 \begin{center}
  \psfrag{x}{$x$}
  \psfrag{y}{$y$}
  \psfrag{1}{$1$}
  \psfrag{34}{$\frac 34$}
  \psfrag{-34}{$-\frac 34$}
  \psfrag{54}{$\frac 54 + \frac \varepsilon 2$}
  \psfrag{74}{$-\frac 74 - \frac \varepsilon 2$}
  \leavevmode\epsfbox{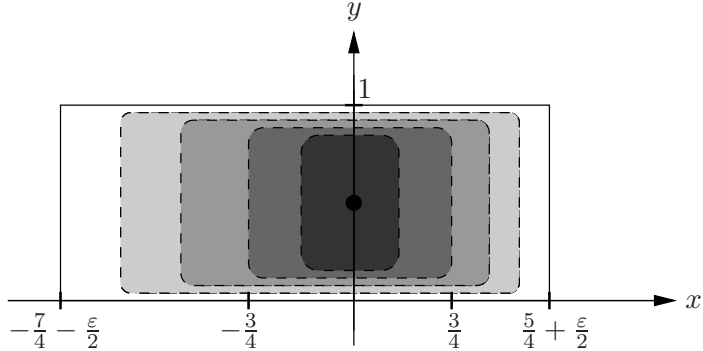}
 \end{center}
 \caption{The map $\sigma_u$ for a function $u_+$ with two pieces.}
 \label{fig.diam3}
\end{figure}

Let $u_+$ and $v_+$ be functions as above, 
and 
$$
\sigma_u \colon D(a) \to \RR^2(x_1,y_1), \quad \sigma_v \colon D(a) \to \RR^2(x_2,y_2)
$$
be symplectic embeddings
associated to $u$ and~$v$ as above.
Then, as in Lemma~\ref{l:alpha}, the product $\si_u\times \si_v$ induces a symplectic embedding
$B^4(a)\to \R^2(\bx)\times (0,1)^2$.  Because $u_+,v_+$ are piecewise linear, 
the image $(\sigma_u \times \sigma_v) (B^4(a))$ of the ball projects to a ($\varepsilon$-neighbourhood of a) polygon in $\R^2(\bx)$ 
whose vertices are determined by the non-smooth points of~$u$ and~$v$. 
For example, if $u_+(\rho) = v_+(\rho) = \frac \rho 2$ then the image is the standard diamond 
$\Diamond(a)$ constructed above.

\m
\ni
Assume now that $a>1$.
Define $d$ by $2d=a-1$,
and suppose that the functions $u_+$, $v_+$ have only two pieces, 
where $u_+$ is standard (i.e.\ equal to $\frac \rho 2$) on $[0,2d]$ and $v_+$ is standard on~$[0,1]$.  Then the distortion occurs when either 
$|x_1| > d$ or $|x_2| > \frac 12$.
%
Computing as in the proof of Lemma~\ref{l:alpha} (again omitting $\gve> 0$), we find that the image of~$B^4(a)$
under $\sigma_u \times \sigma_v$ is contained in $\Diamond \times (0,1)^2$,
where $\Diamond$ is as in Figure~\ref{fig.dist}~(II).




\begin{figure}[h]
 \begin{center}
  \psfrag{x1}{$x_1$}
  \psfrag{x2}{$x_2$}
  \psfrag{a}{$\frac a2$}
  \psfrag{-a}{$-\frac a2$}
  \psfrag{1}{$\frac 12$}
  \psfrag{-1}{$-\frac 12$}
  \psfrag{d}{$d$}
  \psfrag{-d}{$-d$}
  \psfrag{u+}{$u_+(a)$}
  \psfrag{u-}{$u_-(a)$}
  \psfrag{v+}{$v_+(a)$}
  \psfrag{v-}{$v_-(a)$}
  \psfrag{I}{(I)}
  \psfrag{II}{(II)}
  \leavevmode\epsfbox{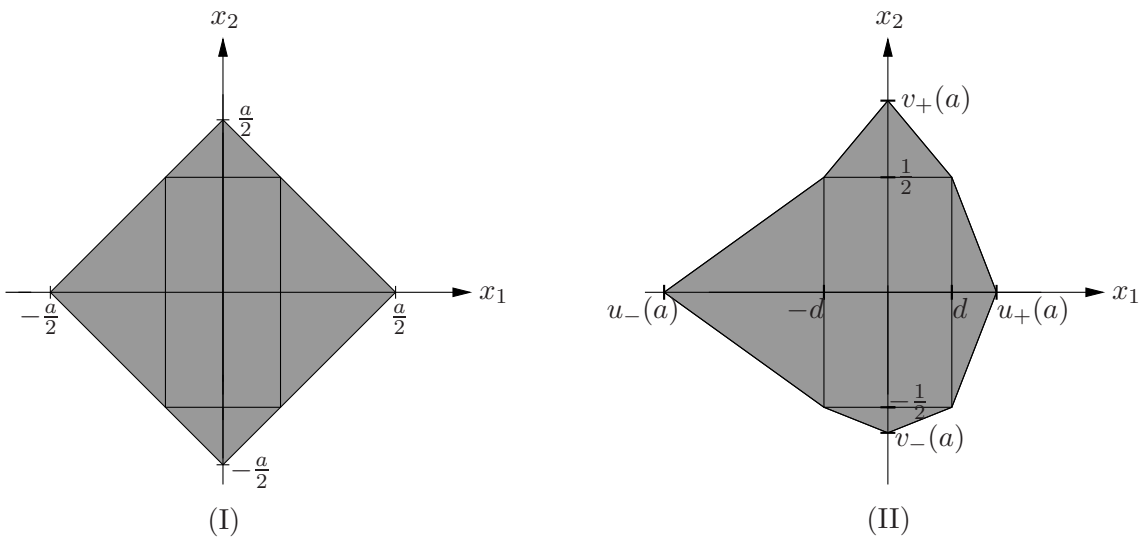}
 \end{center}
 \caption{The diamond $\Diamond (a)$, and a distorted diamond.}
 \label{fig.dist}
\end{figure}

We call the factor $\Diamond$ of such an image $\Diamond \times (0,1)^2$
a {\it distorted diamond}.
A distorted diamond of 
size~$a>1$ therefore consists of 
\begin{itemize}
\item 
a rectangle $(-d,d) \times (-\frac 12,\frac 12)$ with $2d=a-1$,
\item 
a {\bf top and bottom triangle} each with base~$2d$, 
  the sum of whose heights is~$v_+(a)-v_-(a)-1 = a-1=2d$,
  and 
\item 
two {\bf flaps} each with height~$1$, the sum of whose widths is~$u_+(a)-u_-(a)-2d = a-(a-1)=1$.
\end{itemize}

Corollary~\ref{cor:alpha} has the following generalization.

\begin{proposition} \label{prop:symp}
Let $\Diamond$ be a distorted diamond of size~$a$.
Then the product $\Diamond \times (0,1)^2$ is symplectomorphic to the open ball~$B^4(a)$.
\end{proposition}

\proof
In view of the above construction,
for each compact subset $K \subset \Diamond \times (0,1)^2$
there exists $\hat a < a$ and a symplectic embedding 
$\hat \varphi \colon B^4(\hat a) \to \Diamond \times (0,1)^2$
such that $\Ima \hat\varphi \supset K$.
The proposition therefore again follows from Lemma~\ref{l:nest}.
\proofend

In~\cite{Tr}, Traynor used a different construction to prove Proposition~\ref{prop:symp} for the special case that $\Diamond$
is the standard simplex $\left\{ (x_1,x_2) \mid x_1,x_2 >0,\, x_1+x_2 < a \right\}$.

\subsection{Shears} \label{ss:shears}

Let $f \colon \RR \to \RR$ be a smooth function.
Consider the $x_1$-shear $$
\varphi (x_1,x_2) = (x_1 + f(x_2), \, x_2)
$$
 of~$\RR^2$.
Then the diffeomorphism
\begin{equation} \label{e:x1-shear}
\widehat \varphi (x_1,x_2,y_1,y_2) \,=\, 
\bigl( x_1 + f(x_2), \, x_2, \, y_1, \, y_2 - f'(x_2)\,y_1 \bigr) 
\end{equation}
is a symplectomorphism of $\RR^4$.
Indeed, this is just the ``cotangent map'' 
$$
(x_1,x_2,y_1,y_2) \mapsto 
\left( \varphi (x_1,x_2), \bigl( d \varphi (x_1,x_2) )^{T} \bigr)^{-1} (y_1,y_2) \right)
$$
of the shear $\varphi$.
We call a map $\widehat \varphi$ of the form~\eqref{e:x1-shear} also an $x_1$-shear.
Similarly, an $x_2$-shear $\varphi (x_1,x_2) = (x_1, \, x_2+g(x_1))$ induces a symplectomorphism 
\begin{equation} \label{e:x2-shear}
(x_1,x_2,y_1,y_2) \,\mapsto\, 
\bigl( x_1, \, x_2+g(x_1), \, y_1 - g'(x_1) \, y_2, \, y_2 \bigr) , 
\end{equation}
which we again call an {\it $x_2$-shear}.

Let $U \subset \RR^2(\bx)$ be a domain,
and consider the image $\widetilde \varphi (U \times (0,1)^2)$ of an $x_1$-shear
in $\RR^2(\bx) \times \RR^2(\by)$.
This image fibers over $\varphi (U) \subset \RR^2(\bx)$, with fiber 
$\{ (y_1, y_2-f'(x_2)y_1) \mid (y_1,y_2) \in (0,1)^2\}$ 
over $\varphi (x_1,x_2)$,
see Figure~\ref{fig.fiber}.

%
\begin{figure}[h]
 \begin{center}
  \psfrag{y1}{$y_1$}
  \psfrag{y2}{$y_2$}          
  \psfrag{1}{$1$}
  \psfrag{-f}{$-f'(x_2)$}
  \leavevmode\epsfbox{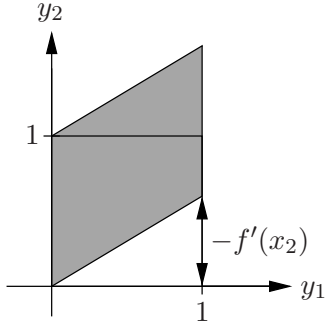}
 \end{center}
 \caption{The fiber over $\varphi (x_1,x_2)$.}
\label{fig.fiber}
\end{figure}

The projection $\pi_{\mathbf y} \colon \RR^2(\by) \to \RR^2(\by) / \ZZ^2(1,1)$
is injective on these fibers.
Further 
$$
T(\mu, 1)\ =\ \RR^2(\bx) / \ZZ^2(\mu,1) 
\ \times \
\RR^2(\by) / \ZZ^2(1,1).
$$
It follows that if the projection 
$\pi_{\bx} \colon \RR^2(\bx) \to \RR^2(\bx) / \ZZ^2(\mu,1)$
is injective
on~$\varphi (U)$, then also 
$$
\pi = \pi_{\bx} \times \pi_{\by}\ \colon \ 
\widetilde \varphi \bigl( U \times (0,1)^2 \bigr) \to T(\mu, 1) 
$$
is injective.
The same holds true for $x_2$-shears.

One can check by direct calculation that an arbitrary composite of shears can map 
$U\times (0,1)^2$ to a set that intersects the fibers $\bx\times \R^2(\by)$ in subsets 
that no longer project injectively under~$\pi_{\by}$.  
The next lemma gives conditions under which shears may be composed:
One essential condition is that the shears must affect disjoint subsets of~$U$.
  
\begin{lemma} \label{l:twoshear}
Suppose that an $x_1$-shear $\varphi_1$ and an $x_2$-shear~$\varphi_2$ 
satisfy the following conditions:
\begin{itemize}
\item[(i)]
$\varphi_2$ fixes the set 
$\varphi_1 \bigl( \{ \bx \in U \mid \varphi_1 (\bx) \neq \bx \} \bigr)$
pointwise,
%
\item [(ii)]
$\pi_{\bx}$ injects $\varphi_2 \circ \varphi_1 (U)$ into 
$\RR^2(\bx) / \ZZ^2(\mu,1)$.
\end{itemize}
Then $\pi$ injects $\widehat \varphi_2 \circ \widehat \varphi_1 \bigl( U \times (0,1)^2 \bigr)$ 
into $T(\mu,1)$.
\end{lemma}

\begin{proof} 
The first condition implies that each $2$-plane $\bx \times \R^2(\by)$ is moved by at most one 
of the shears. Hence each $\bx$-fiber of the image
$\widehat \varphi_2 \circ \widehat \varphi_1 (U \times (0,1)^2)$ 
projects injectively under $\pi_{\by} \colon \R^2(\by)\to \R^2(\by)/\Z(1,1)$.   
It remains to check that the projection to $\R^2(\bx) / \Z(\mu,1)$ is injective, 
which is guaranteed by~(ii).
\end{proof}

\m
We next give three examples illustrating the above embedding method.

\m \ni
\begin{example} \label{ex:1} 
{\it A full filling of $T(2k^2,1)$ for each $k \in \NN$.}
\rm\ \
We start from the diamond $\Diamond (2k)$, 
with vertices $(\pm k, 0)$, $(0, \pm k)$.
Using the linear shear $\varphi(x_1,x_2)=(x_1+(2k-1)x_2,x_2)$, 
the diamond is transformed into the parallelogram~$P(k)$ 
with vertices $(\pm k, 0)$ and $\pm (2k^2-k,k)$, 
see~Figure~\ref{fig.rhomboid}.

\begin{figure}[ht]
 \begin{center}
   \psfrag{k}{$k$}
   \psfrag{-k}{$-k$}
   \psfrag{x1}{$x_1$}
   \psfrag{x2}{$x_2$}
   \psfrag{kk}{$k^2$}
   \psfrag{2kk}{$2k^2-k$} 
   \psfrag{-2kk}{$-2k^2 + k$}
 \leavevmode\epsfbox{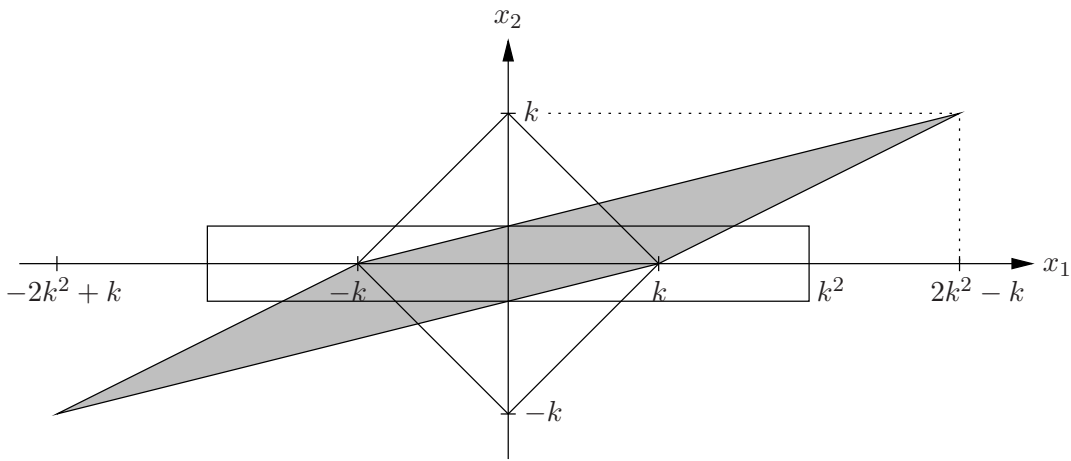}
 \end{center}
 \caption{The parallelogram~$P(k)$.}
 \label{fig.rhomboid}
\end{figure}

\ni
This shear is chosen so that 
\begin{itemize}
\item the vertical distance between the top and bottom edges of~$P(k)$ is~$1$, and
\item each of these edges projects to an interval of length~$2k^2$ on the $x_1$-axis.
\end{itemize}

\begin{figure}[ht]
 \begin{center}
   \psfrag{X1}{$x_1$}
   \psfrag{X2}{$x_2$}
 \leavevmode\epsfbox{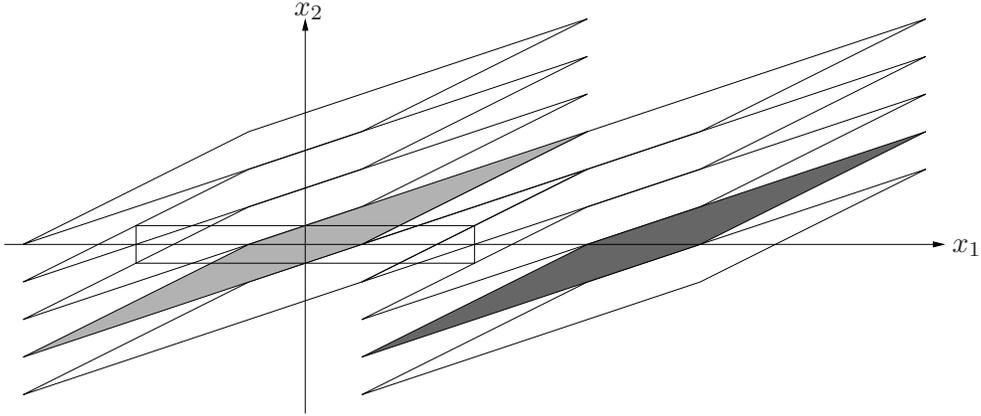}
 \end{center}
 \caption{The tiling of $\RR^2(\mathbf x)$ by translates of $P(k)$ (in the picture $k=3$). The darker parallelogram is a translate of the lighter one by $2k^2$ in the $x_1$-direction. The rectangle marks the standard fundamental domain already depicted in Figure~\ref{fig.rhomboid}.}
  \label{fig.tiling}
\end{figure}
%
It follows that the
set $P(k)$ is a fundamental domain for the action of $\ZZ^2$ with generators 
$2k^2 \partial_{x_1}$ and~$\partial_{x_2}$, cf. Figure~\ref{fig.tiling}. 
\ni Now Proposition~\ref{prop:symp} shows that the ball~$B^4(2k)$
symplectically embeds into~$T(2k^2,1)$. 
\end{example}

\begin{remark} \label{rem:expl}
{\rm 
Together with scaling and Remark~\ref{r:normalform}, this gives 
an explicit full filling by one ball of~$T(\mu,1)$ for all $\mu = \frac{2m^2}{n^2}$ 
with $m,n$ relatively prime.
(Non-explicit full fillings of these tori follow from  
the computation of Seshadri constants~\eqref{eq:seshbound_2d_square}, 
and from Proposition~\ref{p:muge2},
together with Remark~\ref{r:normalform}.)
}
\end{remark}

\begin{remark}
{\rm
By shears as in Example~\ref{ex:1} one can also construct explicit full fillings for some special 
irrational tori, namely for those of the form $\R^4/\Lambda$ where the lattice splits as 
$\Lambda=\Lambda_x \times \Lambda_y$, such that $\Lambda_y$ is the standard $\ZZ^2 \subset \RR^2(\by)$ 
and such that some linearly sheared diamond in~$\RR^2(\bx)$ is a fundamental domain for~$\Lambda_x$. 
}
\end{remark}

%
\begin{figure}[h]
 \begin{center}
  \psfrag{x1}{$x_1$}
  \psfrag{x2}{$x_2$}
  \psfrag{A}{$A$}
  \psfrag{B}{$B$}
  \psfrag{C}{$C$}
  \psfrag{A'}{$A'$}
  \psfrag{a}{$\frac a2$}
  \psfrag{a4}{$\frac a2=\frac 23$} 
  \psfrag{-a}{$-\frac a2$}
  \psfrag{23}{$\frac 13$}
  \psfrag{11}{$(\frac 12, \frac 12)$}
  \psfrag{I}{(I)}
  \psfrag{II}{(II)}
  \leavevmode\epsfbox{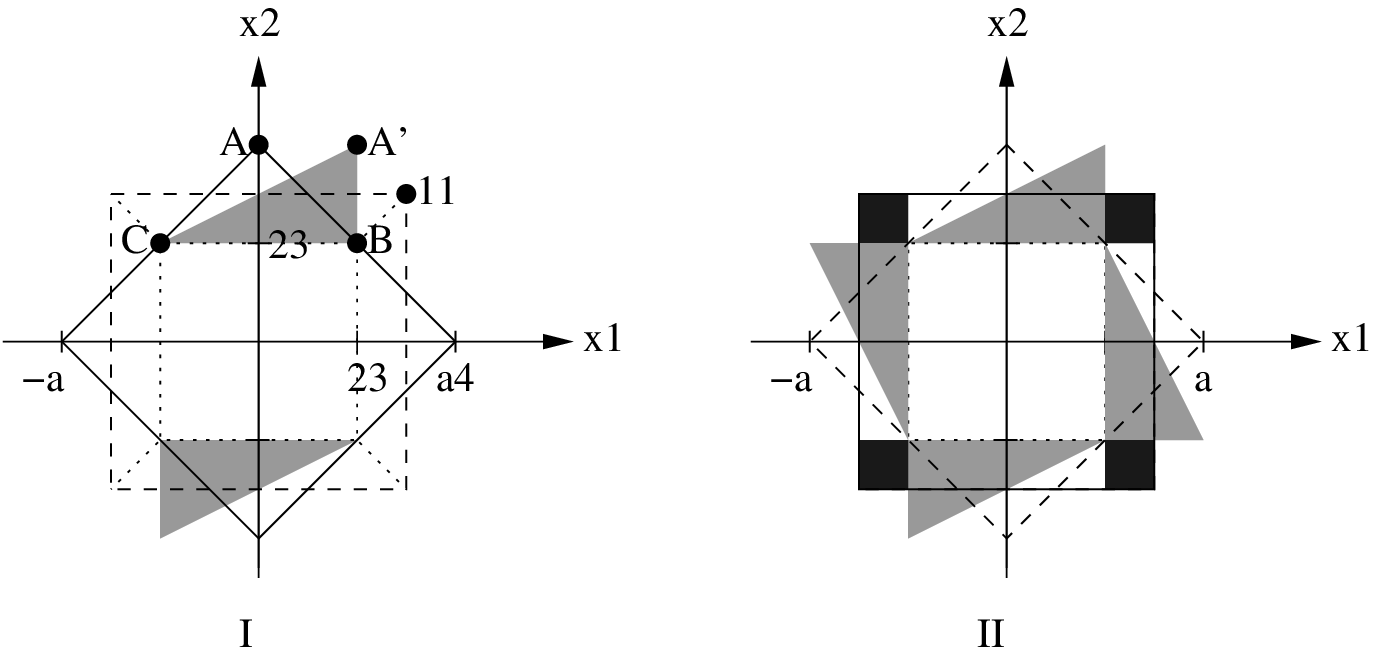}
 \end{center}
 \caption{Filling $\frac 89$ of $T(1,1)$.}
\label{fig43}
\end{figure}

\m \ni
\begin{example} \label{ex:2}
{\it  Filling $\frac 89$ of $T(1,1)$.} \rm \ \ 
Let $a = \frac 43$. The corners of the maximal inscribed square in~$\Diamond (a)$ 
with sides parallel to the axes have coordinates $(\pm \frac 13, \pm \frac 13)$. 
Two of these corners are labelled~$B$ and~$C$ in Figure~\ref{fig43}~(I).
Choose $f$ with $f(x_2)=0$ for $-\frac 13 \le x_2 \le \frac 13$ and $f'(x_2) =1$ for $|x_2| >\frac 13$
(with rounding between).
In particular, $f(-\frac a2) = -\frac 13$ and $f(\frac a2) = \frac 13$.
The $x_1$-shear induced by $f$ then takes the upper triangle~$ABC$ 
to the shaded triangle~$A'BC$, and similarly for the bottom triangle, 
while the rest of the diamond is untouched.
The $x_2$-shear induced by the function $-f$ then moves the left and right flaps
of~$\Diamond (a)$ to the flaps shown in Figure~\ref{fig43}~(II),
while the rest of the image of the $x_1$-shear is untouched.
Therefore, a point in $\Diamond (a)$ is affected by at most one of these two shears, so that we can apply Lemma~\ref{l:twoshear}.

The composition of these two shears takes (a slight shrinking of) $\Diamond (a)$ 
to the shaded domain in Figure~\ref{fig43}~(II).
This domain injects into $\RR^2(\bx) / \ZZ^2(1,1)$:
It wraps up under the action of $\ZZ^2$ to a set covering all of the square $(-\frac 12,\frac 12)^2$ except the four black squares of area~$(\frac 16)^2$ each.
For each $\gve >0$ we thus have constructed a symplectic embedding of a ball into~$T(1,1)$ 
filling at least $\frac 89-\gve$ of the volume of~$T(1,1)$.
\end{example}

\m \ni
\begin{example} \label{ex:3} 
{\it Filling $\frac{49}{50}$ of $T(1,1)$.} 
{\rm Let $a = \frac 75$.
The idea is to divide the square representing $T(1,1)$ into two
rectangles, one the maximum rectangle of height~$1$ that lies in the 
diamond~$\Diamond (a)$ (and hence has width $a-1=\frac 25$),
and the other of width $2-a=\frac 35$, see Figure~\ref{fig75}~(I).

%
\begin{figure}[h]
 \begin{center}
  \psfrag{x1}{$x_1$}
  \psfrag{x2}{$x_2$}
  \psfrag{a}{$\frac a2$}
  \psfrag{1}{$\frac 12$}
  \psfrag{-1}{$-\frac 12$}
  \psfrag{2}{$1$}
  \psfrag{45}{$\frac 25$}
  \psfrag{65}{$\frac 35$}
  \psfrag{-15}{$-\frac 1{10}$}
  \psfrag{I}{(I)}
  \psfrag{II}{(II)}
  \leavevmode\epsfbox{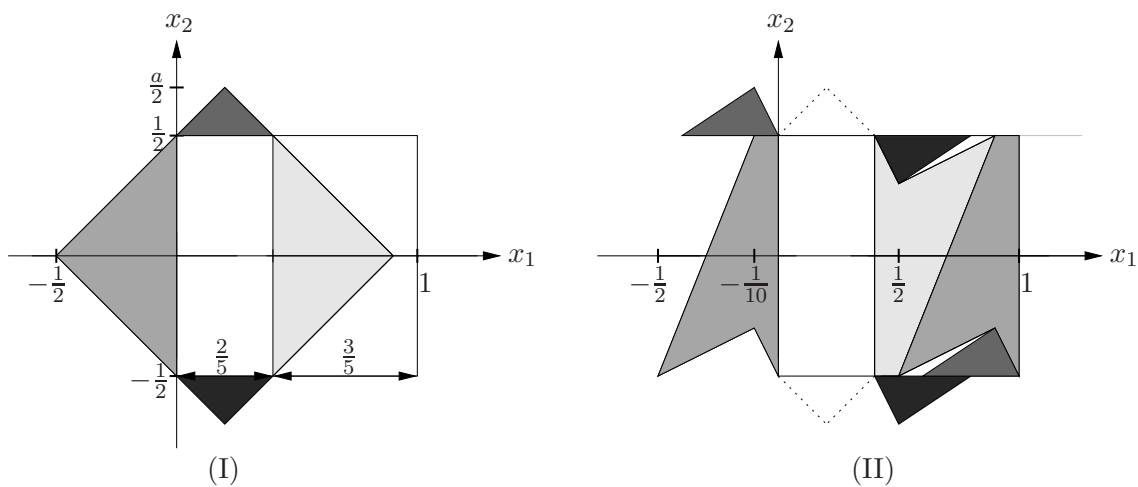}
 \end{center}
 \caption{Filling $\frac{49}{50}$ of $T(1,1)$, schematically.}
 \label{fig75}
\end{figure}
%
%
We shear the top triangle by a strong $x_1$-shear to the left,
the bottom triangle by a strong $x_1$-shear to the right,
and then shear the flaps by $x_2$-shears in a symmetric way 
so as to free triangles into which the sheared top and bottom triangles fit,
when projected to the torus.
The freed triangles have height $\frac 15$ and width~$\frac 12$,
while the triangles fitting in have the same height, but width~$\frac 25$ only.
In Figure~\ref{fig75}~(II), one sees the image of the dark grey
top triangle and its translate by $\partial_{x_1}- \partial_{x_2}$, 
the image of the black bottom triangle and its translate by $\partial_{x_2}$,
as well as the image of the mid-grey
left flap and its translate by $\partial_{x_1}$.
%
%
\begin{figure}[h]
 \begin{center}
  \psfrag{x1}{$x_1$}
  \psfrag{x2}{$x_2$}
  \psfrag{-12}{$-\frac 12$}
  \psfrag{ae}{$\frac{a-\varepsilon}2+\frac 12$}
  \psfrag{1}{$1$}
  \psfrag{25}{$\frac 15$}
  \psfrag{45}{$\frac 25$}
  \leavevmode\epsfbox{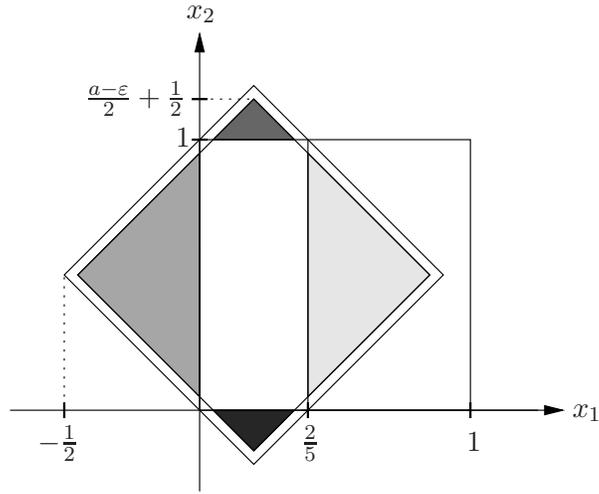}
 \end{center}
 \caption{The decomposition of the diamond $\Diamond (a-\varepsilon)$.}
 \label{fig.diameps}
\end{figure}

To make this construction precise, fix a small $\varepsilon >0$,
and decompose the diamond~$\Diamond (a-\varepsilon)$ into four triangles and a rectangle
of height~$1$ and width~$\frac 25$
from each of whose four vertices a simplex of width $\frac{\varepsilon}2$ has been removed.
For notational convenience, we also translate $\Diamond (a-\varepsilon)$ by $\frac 12 \partial_{x_2}$
(see Figure~\ref{fig.diameps}).

%
\begin{figure}[h]
 \begin{center}
  \psfrag{x1}{$x_1$}
  \psfrag{x2}{$x_2$}
  \psfrag{1}{$1$}
  \psfrag{12}{$\frac 12$}
  \psfrag{ae}{$\frac{a-\varepsilon}2+\frac 12$}
  \psfrag{45}{$\frac 25$}
  \psfrag{A}{$A$}
  \psfrag{B}{$B$}
  \psfrag{C}{$C$}
  \psfrag{D}{$D$}
  \psfrag{C'}{$C'$}
  \psfrag{D'}{$D'$}
  \psfrag{X}{$X$}
  \psfrag{Y}{$Y$}
  \psfrag{W}{$W$}
  \psfrag{Z}{$Z$}
  \psfrag{W'}{$W'$}
  \psfrag{Z'}{$Z'$}
  \leavevmode\epsfysize=10cm\epsfbox{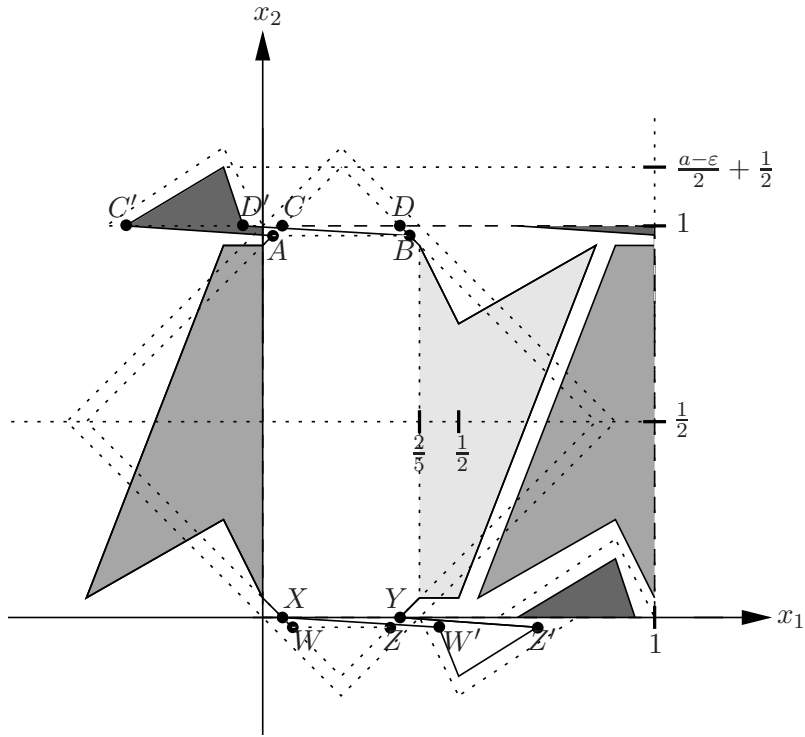}
 \end{center}
 \caption{Filling $\frac{49}{50}$ of $T(1,1)$.}
 \label{fig.precise}
\end{figure}

\ni
The $x_1$-shear, that moves the top triangle to the left and the bottom 
triangle to the right, has the following properties:
\begin{itemize}
\item
It has support in $\{ x_2 > 1-\frac \varepsilon 4\} \cup \{ x_2 < 0 \}$.
\item
Near the upper triangle, 
this shear is very strong on $\{ 1 - \frac \varepsilon 4 < x_2 < 1 \}$,
and up to a minor shear, it is a translation on $\{ 1 < x_2 < \frac{a-\varepsilon}2+\frac 12 \}$. 
\item
It fixes the points 
$A=(\frac \varepsilon 4, 1-\frac \varepsilon 4)$ 
and~$B=(\frac 25-\frac \varepsilon 4, 1-\frac \varepsilon 4)$,
but translates $C = (\frac \varepsilon 2, 1)$ to $C' =(-\frac 25 + \frac \varepsilon 2, 1)$ and 
$D = (\frac 25 - \frac \varepsilon 2, 1)$ to $D' =(-\frac \varepsilon 2, 1)$.
\item
On the bottom triangle, it acts very strongly on 
$\{ -\frac \varepsilon 4 <x_2<0 \}$.
\item
It fixes the points
$X= (\frac \varepsilon 2, 0)$ and~$Y=(\frac 25-\frac \varepsilon 2,0)$,
but translates 
$W=(\frac 34 \varepsilon, -\frac \varepsilon 4)$ to 
$W'=(\frac 25+\frac \varepsilon 2, - \frac \varepsilon 4)$
and $Z=(\frac 25-\frac 34 \varepsilon, - \frac \varepsilon 4)$ to 
$Z'=(\frac 45-\varepsilon, -\frac \varepsilon 4)$. 
\end {itemize}
%
In Figure~\ref{fig.precise} we drew the projection of the (dark grey) 
top triangle $\{ 1-\frac \varepsilon 4 < x_2\}$ to the fundamental domain 
$(0,1) \times (0,1)$ of the usual $\ZZ^2$-action,
but we did not draw the projection of the bottom triangle $\{ x_2<0 \}$.  
In order to see that the image projects injectively to $T(1,1)$,
notice that 
\begin{itemize}\item[-] 
the point~$C$ lies strictly above the segment $\overline{D'B}$,
\item[-]
 $W'$ lies on the right of~$B$,
and 
\item[-] 
$B$ lies $\frac \varepsilon 4$ below~$C$
and $W'$ lies $\frac \varepsilon 4$ below~$X$.
\end{itemize}
Therefore, the translate by $\partial_{x_2}$ of the segment $\overline{XW'}$
lies above the segment $\overline{D'B}$.
}
\end{example}

\section{Proof of Theorem~\ref{t:1}} \label{s:t1}

\ni
Since we already proved Theorem~\ref{t:2}, it only remains to treat the product torus $T(1,1)$. 
So let $a=\sqrt{2}$.
We want to find, for each $\varepsilon >0$, a symplectic embedding of the ball $B^4(a-\varepsilon)$
into~the torus $T(1,1)$.
We describe the schematic embedding for $\varepsilon =0$.
From this, an actual embedding for $\varepsilon>0$ is obtained exactly as in 
Example~\ref{ex:3}.   

As in that example, given the diamond $\Diamond (a)$, 
we decompose the square $(0,1) \times (-\frac 12,\frac 12)$
into two rectangles, and fill the right rectangle 
$(a-1,1) \times (-\frac 12,\frac 12)$
with the four triangles, see Figure~\ref{fig.dist1}~(I).

%
\begin{figure}[h]
 \begin{center}
  \psfrag{x1}{$x_1$}
  \psfrag{x2}{$x_2$}
  \psfrag{1}{$\frac 12$}
  \psfrag{2}{$1$}
  \psfrag{-1}{$-\frac 12$}
  \psfrag{a}{$\frac a2$}
  \psfrag{a-}{$a-1$}
  \psfrag{1-}{$\frac{1-b}2$}
  \psfrag{1+}{$\frac{1+b}2$}
  \psfrag{ht}{$h_t$}
  \psfrag{hb}{$h_b$}
  \psfrag{I}{(I)}
  \psfrag{II}{(II)}
  \leavevmode\epsfbox{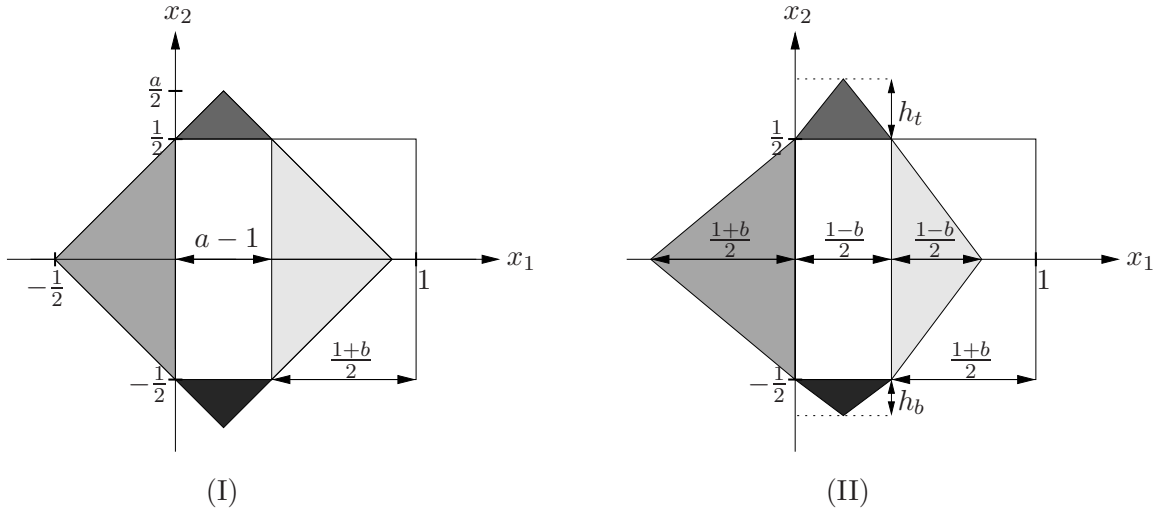}
 \end{center}
 \caption{The decomposition of $\Diamond (a)$, and its distortion.}
 \label{fig.dist1}
\end{figure}

The width of this rectangle is 
$$
1-(a-1) = 2-\sqrt 2 \,=: \frac{1+b}2,\mbox{ where } b := 3-2\sqrt 2.
$$
We start from a distorted diamond as in Figure~\ref{fig.dist1}~(II),  
whose left flap has width~$\frac{1+b}2$,
and so will just fit into the right rectangle.
The right flap of the distorted diamond then has width~$\frac{1-b}2$.
The height~$h_t$ of its top triangle will be chosen later. 
The height~$h_b$ of the bottom triangle is then determined by the fact that the sum of these two heights must be~$\frac{1-b}2$.

Figure~\ref{fig.full} shows the final position of the flaps in the right rectangle.

%
\begin{figure}[h]
 \begin{center}
  \psfrag{x1}{$x_1$}
  \psfrag{x2}{$x_2$}
  \psfrag{2}{$1$}
  \psfrag{1-}{$\frac{1-b}2$}
  \psfrag{1+}{$\frac{1+b}2$}
  \psfrag{2b}{$b$}
  \psfrag{ht}{$h_t$}
  \psfrag{hb}{$h_b$}
  \psfrag{R-}{$T_-$}
  \psfrag{R+}{$T_+$}
  \leavevmode\epsfbox{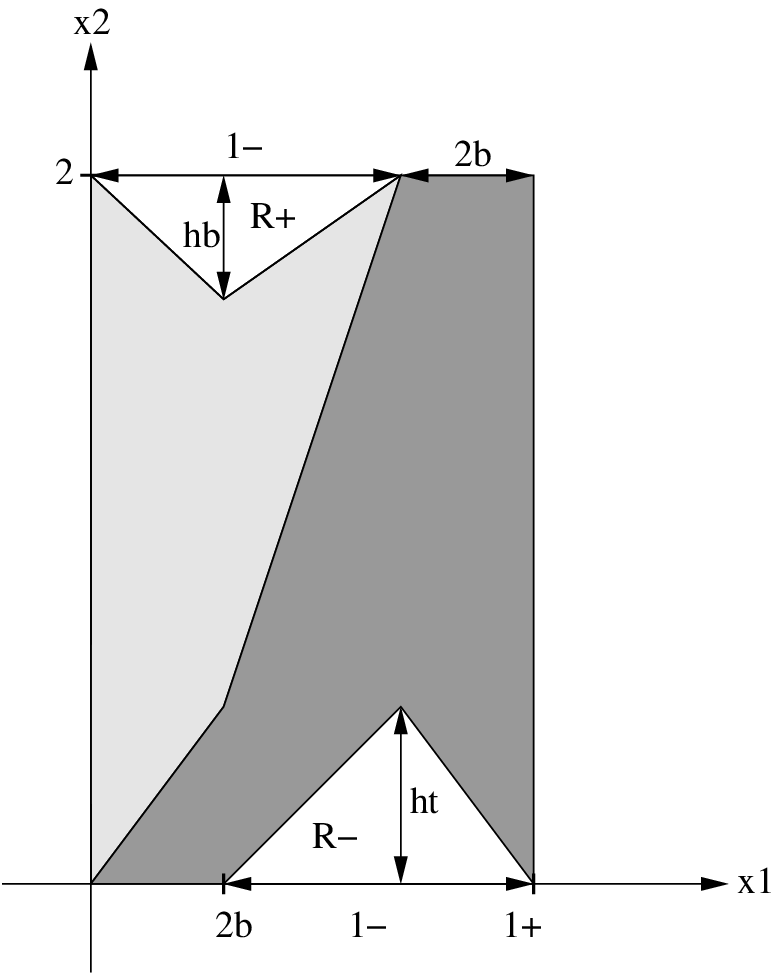}
 \end{center}
 \caption{The final position of the flaps, where we have translated axes so that the origin lies at the corner of the right rectangle.}
\label{fig.full}
\end{figure}

%
\begin{figure}[h]
 \begin{center}
 \psfrag{x1}{$x_1$}
  \psfrag{x2}{$x_2$}
  \psfrag{1}{$1$}
  \psfrag{b}{$b$}
  \psfrag{1-}{$\frac{1-b}{2}$}
  \psfrag{1+}{$\frac{1+b}{2}$}
  \psfrag{2b}{$\frac{2b}{1+b}$}
  \psfrag{2b-}{$\frac{2b}{1-b}$}
  \psfrag{fb}{$\frac{1-b}{1+b}$}
  \psfrag{ht}{$h_t$}
  \psfrag{hb}{$h_b$}
  \psfrag{f1}{$\varphi_1$}
  \psfrag{f2}{$\varphi_2$}
  \psfrag{p1}{$\psi_1$}
  \psfrag{p2}{$\psi_2$}
  \psfrag{l}{$\ell$}
  \leavevmode\epsfbox{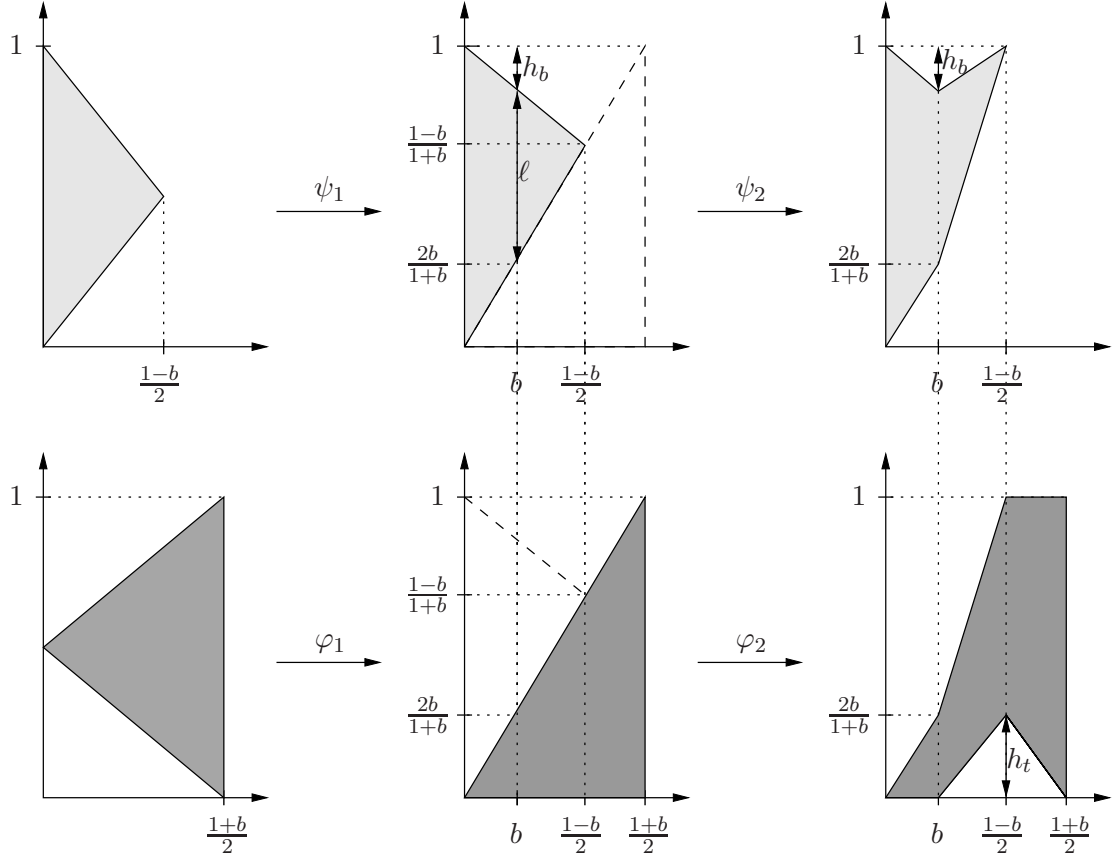}
 \end{center}
 \caption{Filling the right rectangle with the distorted triangles.
	In the middle figure the darker flap has been moved down to fill half the rectangle and the lighter flap   has been correspondingly sheared up.  The right figure shows the effect of a further vertical shear,       leaving two empty triangular regions~$T_\pm$, one at the top and one at the bottom.}
 \label{fig.moves}
\end{figure}

The two remaining triangles~$T_{\pm}$ both have base of length $\frac{1-b}2$, which agrees with that of the triangles in Figure~\ref{fig.dist1}~(II).  
Therefore it remains to check that the heights $h_t, h_b$ of $T_{\pm}$ sum to $\frac{1-b}2$. 
Note that the size of the diamond was chosen so that the total area of the triangles $T_\pm$ left uncovered by the flaps in Figure~\ref{fig.full} equals the sum of the areas of the original (undistorted) top and bottom triangles. Since the base sides of~$T_\pm$ already have the correct length $\frac {1-b}{2}$, this equality of area forces their heights to add up to the correct
amount~$\frac{1-b}{2}$ as well. 
In particular, these heights can be achieved by a suitable distortion.

%
Alternatively, one can explicitly compute the sum of $h_t$ and $h_b$ as follows. Figure~\ref{fig.moves} shows how the flaps are 
sheared vertically so that they fit into this rectangle.
The image of the shear~$\psi_1$ in 
Figure~\ref{fig.moves}~(center top) shows that
$h_b + \ell + \frac{2b}{1+b} = 1$, where $\ell = \frac{1-3b}{1-b}$ is the length of 
the intersection of the lighter flap with the vertical line at~$x_1=b$.  
Further, looking at the shear~$\varphi_2$ we see that $h_t =1-\frac{1-b}{1+b} = \frac{2b}{1+b}$.
Therefore $h_b+h_t = 1-\ell = \frac{2b}{1-b}$.
Hence 
%
%
$h_b+h_t=\frac{1-b}2$ if $\frac{2b}{1-b}=\frac{1-b}2$, 
or equivalently if $b^2 - 6b +1=0$. 
But this holds because we defined $b=3-2\sqrt 2$.

In any case, as
in Example~\ref{ex:3}, 
we can therefore use an $x_1$-shear to bring the top triangle of the distorted diamond 
into~$T_-$ and the bottom triangle into~$T_+$.
\proofend


\section{Proofs of Corollaries~\ref{c:cone} and ~\ref{c:1}} \label{s:cone}

%
\begin{proof}[Proof of Corollary~\ref{c:cone}]
Consider the $4$-torus $T^4$ and its blow-up $\Tilde X = T^4 \sharp \overline{\CC P^2}$.
Fix an orientation of~$\Tilde X$. 
Denote by $E \in H_2(\Tilde X;\ZZ)$ the homology class 
of the exceptional divisor (with some orientation) in~$\Tilde X$.
We need to show that the symplectic cone of~$\Tilde X$ is
\begin{equation} \label{e:inc}
\cc(\Tilde X) \,=\, 
\left\{ \alpha \in H^2(\Tilde X;\RR) \mid \alpha^2>0,\, \alpha(E) \ne 0 \right\} .
\end{equation}

We first prove the inclusion $\subset$ in~\eqref{e:inc}. 
The condition $\alpha^2>0$ holds because $\alpha$ is represented 
by a symplectic form compatible with the given orientation of~$\Tilde X$.
The condition $\alpha (E) \neq 0$ follows from 
Taubes' work on the relation between Seiberg--Witten and Gromov invariants,
according to which for any symplectic form~$\om$ on~$\Tilde X$ the class~$E$ 
is representable by an embedded sphere on which~$\om$ is non-degenerate, 
see~\cite{Taubes.announcement,Taubes.book}. 

\s
We now prove the inclusion $\supset$ in~\eqref{e:inc}.
The projection $\pi \colon \Tilde X \to T^4$ induces an orientation on~$T^4$. 
The classes in $H^2 (\Tilde X;\RR)$ can be written as 
$\pi^* \beta - a \;\!e$ where $\beta \in H^2 (T^4;\RR)$ and $a \in \RR$,
and where $e = \PD (E)$ is the Poincar\'e dual of~$E$.
Since $e^2 = -1$, the set on the right hand side of~\eqref{e:inc} becomes
$$
\left\{ \pi^* \beta -a\;\!e \in H^2(\Tilde X;\RR) \mid 
                       \beta \in H^2(T^4;\RR), \,\beta^2 > a^2 > 0 \right\} .
$$
Fix $\beta \in H^2(T^4;\RR)$ and $a >0$ with $\beta^2 > a^2 >0$.
Since $\beta^2 >0$, we can represent $\beta$ by a linear symplectic form on~$T^4$
compatible with the given orientation.
Since $\beta^2 > a^2 >0$, Theorem~\ref{t:1} guarantees a symplectic embedding 
of $B^4(a)$ into~$(T^4,\omega)$. 
The symplectic form on the corresponding symplectic blow-up of $(T^4,\omega)$ 
represents the class $\pi^*\beta - a' e$, where either $a'=a$ or $a'=-a$.

It is well known (see e.g.~\cite[Lemma 2]{W}) that there exists an orientation preserving diffeomorphism $\varphi \colon \Tilde X \to \Tilde X$ of the blow-up which acts on $H^2(\Tilde X;\R)$ by sending~$e$ to $-e$ while fixing the orthogonal complement $\pi^*\left(H^2(T^4;\R)\right)$ of~$e$ 
(with respect to the cup product pairing).\footnote{Such a diffeomorphism $\varphi$ can be constructed explicitly as follows: 
The map $c \colon \overline{\C P^2} \to \overline{\C P^2}$ given by 
$c([z_0:z_1:z_2]) = [\bar z_0:\bar z_1:\bar z_2]$ is orientation preserving and reverses 
the orientation of all complex lines. 
By an isotopy supported near the fixed point~$p_0=[1:0:0]$ we can deform~$c$ to a diffeomorphism~$c'$ fixing a neighborhood of~$p_0$, and if the connected sum to construct the blow-up is performed in this neighborhood, then $\varphi$ is obtained by glueing~$c'$ to the identity map on the torus.}
%
%
%
In particular, $\varphi^* (\pi^*\beta - a' e) = \pi^*\beta + a' e$.
Therefore, both $\pi^*\beta - a \;\!e$ and $\pi^*\beta + a \;\!e$ belong to 
the symplectic cone~$\cc(\Tilde X)$.
\end{proof}

\begin{proof}[Proof of Corollary~\ref{c:1}]
We denote by $\om$ a symplectic form on~$T^4$ 
such that $(T^4,\om)$ is symplectomorphic to~$T(1,1)$.
Let again $\Tilde X$ be the blow-up of $T^4$, 
and let $E$ be the class of the exceptional divisor~$\Sigma$ (with orientation specified later).
By Corollary~\ref{c:cone}, 
the class $\pi^*[\omega] - a \PD(E)$ admits a symplectic representative for all $0<a<\sqrt{2}$.
We shall show that for $a > \frac 43$, this class admits no K\"ahler representative.

Let $\alpha$ be a K\"ahler form on $\Tilde X$, 
and let $\pi^*[\omega] - a \PD(E)$ be its class. 
By the Enriques--Kodaira classification, the complex manifold $(\Tilde X, \Tilde J)$
underlying the K\"ahler manifold $(\Tilde X, \alpha)$ is the complex blow-up of 
a complex torus~$(T,J)$. 
Orient $\Sigma$ by~$\Tilde J$.
The K\"ahler form~$\alpha$ is positive on all non-constant $\Tilde J$-holomorphic curves in~$\Tilde X$.
In particular, $a = \int_\Sigma \alpha >0$,
and given a non-constant $J$-holomorphic curve~$C$ in~$T$ with proper transform~$\Tilde C$, 
we have $E \cdot [\Tilde C] \ge 0$ and
\begin{equation} \label{e:NM}
0 \,<\, \bigl( \pi^*[\omega]- a \PD(E) \bigr) \bigl( [\Tilde C] \bigr) \,\le\, [\omega] \bigl( [C]\bigr) .
\end{equation}
The Nakai--Moishezon criterion thus implies that $[\omega]$ is an ample class on~$T$.
Hence $[\omega]$ gives a principal polarization of~$T$.
Comparing definition~\eqref{def:Seshadri} with the left inequality in~\eqref{e:NM}
we see that its Seshadri constant is at least~$a$.
Together with Steffens' estimate~\eqref{eq:seshadri_standard_4torus} we thus find $a \le \frac 43$.
\end{proof}

\section{Remarks and Questions} \label{s:quest}

\ni
{\bf 1.  Symplectic forms on $T^4$.}  
We have worked throughout with a linear symplectic form on~$T^4$.  
It is not known whether every symplectic form on~$T^4$ is isotopic to a linear form, 
or even whether it is symplectomorphic to such a form.


\MS
\ni
{\bf 2. Very full fillings.}  
There is a stronger version of full filling: rather than asking whether one can fill 
an arbitrarily large fraction of the volume of a manifold~$M$ with a ball, one could ask whether~$M$ 
has a set of full measure that is symplectomorphic to an open ball. 
In other words, if $a = c_G(M,\om)$ does the open ball~$B^4(a)$ embed symplectically in~$M$?  
Let us say that in this case $(M,\om)$ has a very full filling (by one ball).
(There are similar versions for other filling problems.)
When a rational or ruled manifold has a full filling, it also has a very full filling 
because one can argue as in the proof of Lemma~\ref{l:nest}, using the fact that  
in this case the space of ball embeddings is connected. However, these general arguments 
do not apply to tori, and it is unclear whether~$T(1,1)$, for example, 
has a very full filling by one ball.
On the other hand, the explicit fillings in Example~\ref{ex:1} and Remark~\ref{rem:expl} 
give very full fillings.

\MS
\ni
{\bf 3. The isotopy problem.}
For some symplectic four-manifolds $(M,\omega)$,
such as the complex projective plane or a product of $2$-spheres, 
it is known that the space of symplectic embeddings of a given (closed) ball
into~$(M,\omega)$ is connected, see~\cite{M-blowup, M-isotopy}.
For tori, this is a completely open problem. 
For many balls $B^4(a)$, our embedding constructions yield various symplectic embeddings
into tori~$T(\mu,1)$, for which we do not know whether they are symplectically isotopic.

As a first example, consider,
for some fixed small~$\varepsilon >0$,  
the symplectic embeddings of a ball filling $\frac 89-\varepsilon$ of~$T(1,1)$
that are illustrated, for $\varepsilon =0$, in Figure~\ref{fig:isotopy}.
Here, the embedding $(+-)$ is the one of Example~3 in Section~\ref{ss:shears},
and the other three embeddings are obtained in the same way. 
Are these balls symplectically isotopic in~$T(1,1)$?
Note that, for instance, the (not Hamiltonian) symplectomorphism 
$(x_1, y_1, x_2, y_2) \mapsto (-x_1, -y_1, x_2, y_2)$
of~$T(1,1)$ maps the ball~$(++)$ to $(--)$ and maps $(+-)$ to $(-+)$.

%
\begin{figure}[ht]
 \begin{center}
   \psfrag{x1}{$x_1$}
   \psfrag{x2}{$x_2$}
   \psfrag{++}{$(++)$}
   \psfrag{--}{$(--)$}
   \psfrag{+-}{$(+-)$}
   \psfrag{-+}{$(-+)$}
 \leavevmode\epsfbox{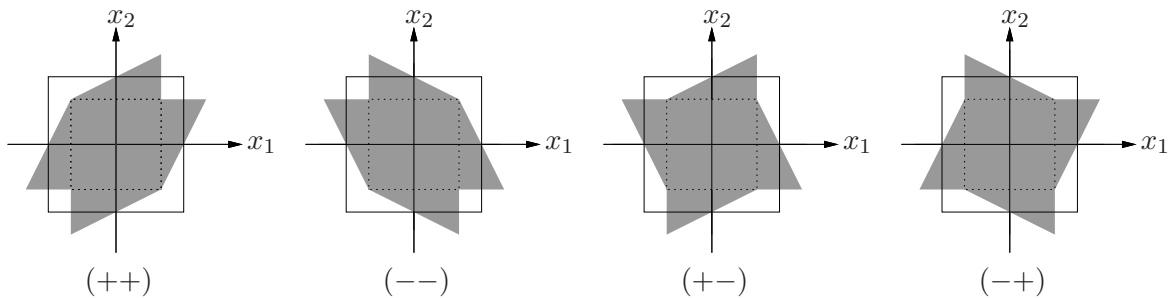}
 \end{center}
 \caption{Four embeddings of $B^4(\frac 43)$ into $T(1,1)$.}
  \label{fig:isotopy}
\end{figure}
%
%

As a second example, consider the following two full fillings of $T(\frac 98,1)$:
The first filling is the one obtained from the explicit full filling of~$T(1,72)$
via Lemma~\ref{lem:type}.
The second filling is similar to the embedding in Example~3 of Section~\ref{ss:shears}.
We decompose the diamond~$\Diamond (\frac 32)$ and the rectangle $(0,\frac 98) \times (0,1)$ 
as in~Figure~\ref{fig.98}~(I), and shear the triangles as shown 
in~Figure~\ref{fig.98}~(II).\footnote
{
In fact, this is the first in a family of full fillings of
$T \bigl( \frac{(2k+1)^2}{2(k+1)^2}, 1 \bigr)$, $k \ge 1$,
by the diamond $\Diamond(\frac{2k+1}{k+1})$
in which the top triangle $x_2 \ge \frac 12$
is sliced into $k$  horizontal slices of
heights
       $\frac{1}{(k+1)^2},\frac{2}{(k+1)^2},\dots,\frac{k}{(k+1)^2}$
and then
sheared to the right
so that the right edge of the $j$th piece lies over an
$x_1$-interval of length
      $\frac{2j-1}{2(k+1)^2}$,
while the bottom triangle is sheared symmetrically
to the left.
}
These two embeddings are clearly different. Are they symplectically isotopic?

%
\begin{figure}[ht]
 \begin{center}
   \psfrag{x1}{$x_1$}
   \psfrag{x2}{$x_2$}
   \psfrag{1}{$\frac 12$}
   \psfrag{-1}{$-\frac 12$} 
   \psfrag{2}{$1$}
   \psfrag{32}{$\frac 34$}
   \psfrag{94}{$\frac 98$} 
   \psfrag{-14}{$-\frac 18$}
   \psfrag{54}{$\frac 58$}  
 \leavevmode\epsfbox{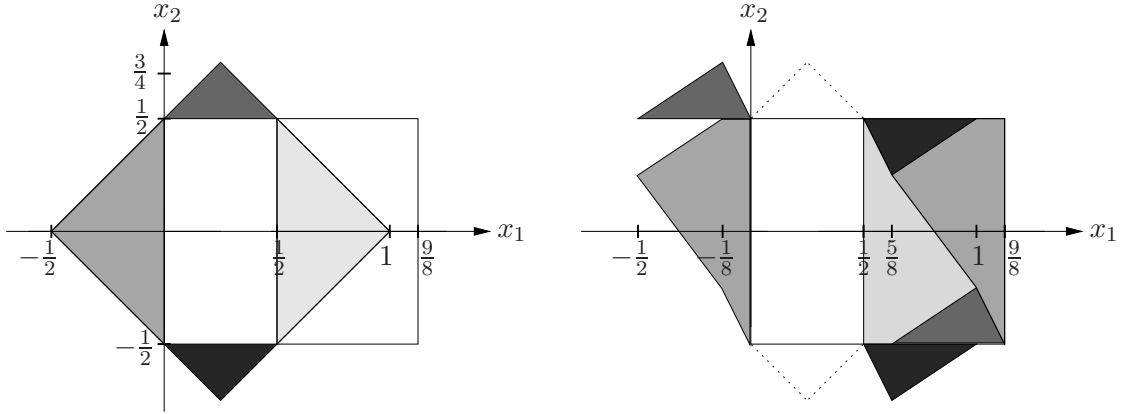}
 \end{center}
 \caption{Another full filling of $T(\frac 98,1)$.}
 \label{fig.98}
\end{figure}
%
%

More generally, 
it is not known whether there is any $\varepsilon \in (0,\sqrt 2)$  such that the space of symplectic embeddings of~$B^4(\varepsilon)$ into~$T(1,1)$ is connected.

\medskip \ni
{\bf 4. Uniqueness of symplectic structures on the blow-up of~$T^4$.}
Suppose that $\om_1,\om_2$ are two cohomologous symplectic forms on the blow-up 
of a given symplectic $4$-manifold~$(M,\om)$
that are obtained by blow-up from two ball embeddings into~$M$. 
Then it is shown in~\cite[Prop.~7.20]{MS} that $\om_1, \om_2$
are isotopic if and only if the two ball embeddings are symplectically isotopic.
Moreover, 
the easiest way to show that two blow-ups of some symplectic $4$-manifold are symplectomorphic 
is to use this equivalence. 
(Of course sometimes, as with the $(++), (- -)$ embeddings mentioned in 4.\ above, 
there are obvious symplectomorphisms that take one embedding to another and hence one blow-up to the other.)
Since the symplectic isotopy problem 
for embeddings of balls in $4$-dimensional tori is open,
the uniqueness problem for symplectic forms on their blow-ups is open too, 
even if we restrict consideration to forms on the blow-up that blow down to linear forms.

\medskip \ni
{\bf 5. Higher dimensions.}
The filling methods used in \S\ref{s:pack} work only in dimension~$4$.
Although many of the explicit arguments in~\S\ref{s:basic} extend to higher dimensions,
the higher dimensional analogs of the diamond~$\Diamond$ (e.g.\/ the octahedron) 
do not tile Euclidean space. 
Therefore there seem to be no simple explicit full fillings of tori by balls 
in higher dimensions along the lines of Example~\ref{ex:1}.
As we explained in Section~\ref{ss:higherdim},
one can get some (presumably rather weak) lower bounds 
for the ball filling number of tori of dimension $2n \ge 6$ from
the computations of Seshadri constants in~\cite{BauerSzemberg:threefolds} and~\cite{Debarre:higher}.  
For example, when $n=3,4$ we have:
$$
p(\T^6) \ge \frac{288}{343} \quad \mbox{ and } \quad p(\T^8) \ge \frac 23 .
$$  
It is not clear how to do better than this, 
or even how to realize these bounds by explicit embeddings.
It is also not clear how to find explicit embeddings in dimension~$4$ that do better than 
some of the sharper Seshadri constants, for example $\frac {360}{361}$ for~$T(1,5)$.

%

\medskip \ni
{\bf 6. Packings by cubes and polydiscs.}
Instead of looking at symplectic embeddings of balls,
one may study symplectic embeddings of cubes and polydiscs~$B^2(a_1) \times \dots \times B^2(a_n)$. 
(In fact, the problem of symplectically packing certain domains in~$\RR^{2n}$ 
by equal cubes has been an original motivation to consider symplectic packings, 
see~\cite{FP,MP}.)
This problem has been much less studied, since (to our knowledge) even in dimension four 
the problem of embedding a polydisc cannot be reduced to the problem of embedding a
collection of balls.
Ekeland--Hofer capacities, \cite{EH2}, and 
the new 4-dimensional invariants from embedded contact homology, \cite{Hu}, 
provide obstructions for symplectic embeddings of polydiscs into certain manifolds.
A few symplectic embedding constructions for polydiscs are described in~\cite{Sch-book};
e.g., the shears in Section~7.1 show that $T(1,1)$ can be fully filled by $B^2(1/k) \times B^2(k)$
for each $k \in \NN$. 
Similarly, the dense set of product tori $T \bigl(1, \frac{m^2}{n^2} \bigr)$
with $m,n \in \NN$ relatively prime can be fully filled by a cube.
Indeed, by Remark~\ref{r:normalform}, it suffices to fill $T(1,k^2)$ by a cube
for each $k \in \NN$. 
For this, view $B^2(k) \times B^2(k)$ as $(-\frac k2, \frac k2)^2 \times (0,1)^2 \subset \RR^2(\bx) \times \RR^2(\by)$, and apply 
(similar to the map in Example~\ref{ex:1}) 
the $x_1$-shear defined by
$\varphi (x_1,x_2) = (x_1+kx_2,x_2)$.


\end{document}